\documentclass{amsart}

\usepackage{amsfonts, amsmath,amscd, amssymb, latexsym, mathrsfs, stmaryrd, verbatim, wasysym }
\usepackage{enumerate}
\usepackage{graphicx}
\usepackage[all]{xy}
\newtheorem{theorem}{Theorem}

\newtheorem{definition}{Definition}

\newtheorem{lemma}[theorem]{Lemma}
\newtheorem{proposition}[theorem]{Proposition}

\newtheorem{assumption}[theorem]{Assumption}
\numberwithin{equation}{section}
\numberwithin{theorem}{section}
\numberwithin{definition}{section}

 \usepackage{color}
\definecolor{darkgreen}{cmyk}{1,0,1,.2}
\definecolor{m}{rgb}{1,0.1,1}

\DeclareMathOperator{\supp}{supp}   

\DeclareMathOperator{\trace}{trace} \DeclareMathOperator{\tr}{tr}

\DeclareMathOperator{\Ind}{Ind}
\DeclareMathOperator{\sgn}{sgn}

\DeclareMathOperator{\const}{const}
\DeclareMathOperator{\Tch}{Tch}

\newcommand{\forget}[1]{}

\def  \nuint {\raise10pt\hbox{$\nu$}\kern-6pt\int}

\newcommand\Tr{\operatorname{Tr}}

\def\N{\mathcal N}
\def \L{\mathcal L}
\def \A{\mathcal A}

\def \P{\mathcal P}

\newcommand\G{\mathcal G}
\newcommand\Q{\mathcal Q}
\newcommand\R{\mathcal R}
\newcommand\V{\mathcal V}

\def \L {{\cal L}}
\def \Sp {{\cal S}}
\newcommand\B{\mathcal B}
\newcommand\Bi{\B^\infty}
\def \J{\mathcal J}

\def \Ch {{\rm Ch}}
\def\Id{{\rm Id}}

\newcommand\ch{\operatorname{ch}}
\newcommand\cyl{\operatorname{cyl}}
\newcommand\ha{\frac12}

\newcommand\D{\mathcal D}
\newcommand\Di{D\kern-6pt/}
\newcommand\cDi{{\mathcal D}\kern-6pt/}
\newcommand\spi{S\kern-6pt/}
\newcommand \cspi{\Sp\kern-6pt/}

\newcommand\CC{\mathbb C}

\def \cal {\mathcal}

\newcommand\NN{\mathbb N}

\newcommand\RR{\mathbb R}
\newcommand\ZZ{\mathbb Z}

\newcommand\pa{\partial}
\newcommand\Ker{\operatorname{Ker}}

\def\tM{{\widetilde M}}

\newcommand{\dvd}{\widetilde{V}_{\D}}
\newcommand{\del}{\widetilde{e}_1}
\newcommand{\dvtd}{\widetilde{V}_{\D^\otimes}}

{\catcode`@=11\global\let\c@equation=\c@theorem}

\allowdisplaybreaks[2]
\date{}
 \usepackage{color}
\definecolor{darkgreen}{cmyk}{1,0,1,.2}
\definecolor{m}{rgb}{1,0.1,1}

\title[higher APS index theorem]{A note on the higher Atiyah-Patodi-Singer index theorem \\on Galois coverings}
\author{Alexander Gorokhovsky}
\address{Department of Mathematics, University of Colorado at Boulder}
\email{Alexander.Gorokhovsky@Colorado.EDU}
\author{Hitoshi Moriyoshi}
\address{Graduate School of Mathematics, Nagoya University}
\email{moriyosi@math.nagoya-u.ac.jp}
\author{Paolo Piazza}
\address{Dipartimento di Matematica, Sapienza Universit\`a di Roma}
\email{piazza@mat.uniroma1.it}

\subjclass[2010]{Primary: 58J20. Secondary: 58J22, 58J42, 19K56.}

\keywords{Galois coverings, groupoids,  group cocycles,
index classes, relative pairing, excision, Atiyah-Patodi-Singer higher index theory, higher eta invariants.}

\begin{document}

\begin{abstract}
Let $\Gamma$ be a finitely generated discrete group satisfying the rapid decay
condition.
We give a new proof of the higher Atiyah-Patodi-Singer theorem on a
Galois $\Gamma$-coverings, thus providing an explicit formula for the higher index associated to a group
cocycle $c\in Z^k (\Gamma;\CC)$ which is of polynomial growth with respect to a word-metric. Our new proof employs relative K-theory and relative cyclic cohomology
in an essential way.
\end{abstract}

\maketitle

\section{Introduction}\label{sect:intro}

Among the many  results in index theory that have followed the original
work of Atiyah and Singer, few have been as inspiring and central in the whole
field as the higher
index theorem on Galois $\Gamma$-coverings of Connes and Moscovici \cite{CM}. The theorem
itself can be seen as a far reaching generalization of the family index theorem
of Atiyah and Singer, in the sense that it can be reduced to it when $\Gamma=\ZZ^k$. This fundamental observation, due to Lustzig \cite{Lustzig}, and the heat-kernel proof
of the family index theorem, due to Bismut \cite{Bismut}, are at the basis of a different proof
of the Connes-Moscovici index theorem, which was given by Lott
in \cite{LottI}. This new proof  employs the superconnection formalism, suitably extended to the noncommutative framework,  in an essential way.
The work of Bismut and Lott opened the way to versions of these theorems
on manifolds with boundary, in the spirit of the seminal
work of Atiyah, Patodi and Singer \cite{APS1}. Contributions were given by Bismut-Cheeger
\cite{BC1} \cite{BC2}
and Melrose-Piazza \cite{MPI} \cite{MPII} for families and by Leichtnam and Piazza \cite{LPMEMOIRS},
based on a conjecture of Lott \cite{LottII},  for Galois coverings.
Geometric applications of these index theorems were given in \cite{LPBSMF},
\cite{LLP}, \cite{LPPSC},\cite{PSJNCG}, \cite{Wahl_higher_rho}.

In this article we give a new proof of the higher Atiyah-Patodi-Singer index theorem
on Galois $\Gamma$-coverings. This new proof is based in a crucial way  on the excision isomorphism
in K-theory and on the pairing between relative K-theory and relative cyclic cohomology; the $b$-calculus  of Melrose and his $b$-trace formula
also play an important role. The ideas we employ have been already exploited successfully in \cite{mp},
where a Godbillon-Vey index theorem on foliated bundles with boundary
was established. For more on the use of the  pairing between relative K-theory
and relative cyclic cohomology see also \cite{lmpflaum, LMP2}.
Our task here is to transfer and adapt
the ideas used in \cite{mp}
to the context of Galois
$\Gamma$-coverings, with $\Gamma$ a finitely generated discrete group
satisfying the (PC) and (RD) conditions  (Polynomial Cohomology and
Rapid Decay). Our main result provides a formula of Atiyah-Patodi-Singer type
for the higher index $\Ind_{(c,\Gamma)} (\D)$ associated to $c\in Z^k (\Gamma;\CC)$; here $\D$ is the Mishchenko-Fomenko operator associated
to a $\Gamma$-equivariant  Dirac-type operator $\widetilde{D}$ on the total space of a $\Gamma$-covering with boundary. We assume, as usual, that the associated Dirac
operator on the boundary, $\widetilde{D}_{\partial}$, is $L^2$-invertible.
The higher index $\Ind_{(c,\Gamma)} (\D)$ is obtained by pairing the index class
$\Ind (\D)$ with a suitably defined cyclic cocycle $\tau_c$ associated to $c$
(we shall of course be more precise later, we only want to give the main ideas here);
one
of the main steps in our proof is the production of a relative index class
$\Ind (\D,\D_{\partial})$ and of a relative cyclic cocycle $(\tau^r_c,\sigma_c)$
and the proof of the following equality:
$\Ind_{(c,\Gamma)} (\D)=\langle \Ind (\D,\D_{\partial}),[(\tau^r_c,\sigma_c)]
\rangle$; one crucial technical problem we have to face is the extendability property
for the relative cocycle $(\tau^r_c,\sigma_c)$. It should be noticed
that compared to the original result
in \cite{LPMEMOIRS} our theorem has the advantage of providing the boundary
correction term, i.e. the higher eta invariant $\eta_{(c,\Gamma)} (\D_{\partial})$,
in a more explicit form; indeed, our higher eta invariant comes already paired,
whereas in \cite{LPMEMOIRS} the higher eta invariant is the result of a pairing
between  a rather abstract object, the higher eta invariant of Lott,
$$\eta_{{\rm Lott}} (\D_{\partial})\in
\widehat{\Omega}_* (\B^\infty)/\overline{[\widehat{\Omega}_* (\B^\infty),\widehat{\Omega}_* (\B^\infty)]}$$
and a cyclic cocyle $t_c$ associated to $c$. In fact, an application of our index formula
is a precise expression for the number $\langle \eta_{{\rm Lott}} (\D_{\partial}),t_c \rangle$
appearing in \cite{LPMEMOIRS}.

%
%

The paper is organised as follows. We start in Section \ref{section:geometric}
with a few geometric preliminaries, including a brief discussion
on relative and absolute cyclic (co)homology. We then move on in Section \ref{sect:index}
and define the index class $\Ind(\D)$, see Subsection \ref{subsect:index}; we
express this index class in terms of the Wassermann projector in Subsection  \ref{subsection:other-index-classes}; in Subsection  \ref{subsect:excision}
we define the relative index class $\Ind(\D,\D_{\partial})$ and prove that
corresponds to $\Ind(\D)$ via excision. In Section \ref{section:higher-indeces}
we define the higher indeces and we compare them with the ones defined
by Leichtnam and Piazza in \cite{LPMEMOIRS}, proving that they are in fact equal.
In Section \ref{sect:relative-cocycles} we show how to define a {\em relative}
cyclic cocycle starting from a $c\in Z^k (\Gamma;\CC)$. In the following section, Section
\ref{sect:continuity}, we prove that under the two assumptions (PC) and (RD),
as in Connes-Moscovici \cite{CM}, our relative cocycles are continuous on the relevant algebra. Finally in Section \ref{sect:theorem} we state and prove our main
result, Theorem \ref{theo:main}

\medskip
\noindent
{\bf Acknowledgements.}
Part of this work was done while A.G. and H.M. were visiting the
Department of Mathematics of Universit\`a di Roma "La Sapienza".
Financial support for these visits was provided by the {\it Japan Society for the Promotion of Science (JSPS)}
and by {\it Istituto Nazionale di Alta Matematica "Francesco Severi"}. Further work was performed during visits
of A.G. and P.P. to Universit\'e Paris 7, in the {\it \'Equipe  Alg\`ebres d'op\'erateurs}; these two authors wish to extend their thanks to the \'Equipe Alg\`ebres d'op\'erateurs for the warm hospitality; financial support for these visits was provided  by
{\it
Fondation Sciences math\'ematiques de Paris} (A.G.) and by the {\it Ministero Universit\`a Ricerca Scientifica}
(PRIN {\it Spazi di moduli e teoria di Lie}) (P.P). A. G. was partially supported by NSF grant DMS-0900968.\\
We thank the anonymous referee for thorough reading and useful comments which helped us improve  our article.

\section{Preliminaries}\label{section:geometric}

\subsection{Manifolds with boundary and cylindrical ends}$\;$\\
Let
$(M_0,g_0)$ be a compact even dimensional riemannian manifold with boundary; the metric is assumed to be of product type in a collar
neighborhood $U\cong [0,2]\times \partial M_0$ of
the boundary: thus $g_0$ restricted to $U$ can be written through the above isomorphism as $dt^2 + g_{\partial}$,
with $g_{\partial}$ a riemannian metric on  $\partial M_0$. We consider the associated manifold with cylindrical ends
$M:= M_0\cup_{\partial M_0} \left(   (-\infty,0] \times \partial M_0 \right)$,
endowed with the extended metric $g$.
The coordinate along the cylinder will be denoted by $t$. We will also consider the $b$-version of $(M,g)$, obtained by
performing the change of variable $\log x=t$.
This is a $b$-riemannian
manifold with product $b$-metric $$\frac{dx^2}{x^2}+ g_{\partial}$$
 near the boundary. We shall freely pass from the $b$-picture to the
cylindrical-end picture, without employing two different notations.
(Our arguments will actually apply to the more general case of {\it exact}
$b$-metrics, or, equivalently, manifolds with asymptotic cylindrical ends; we shall not insist on this point.)
\footnote{In this article we basically give the $b$-calculus of Melrose, and its generalisations,  as known; the basic reference
is of course Melrose' book  \cite{Melrose}; short survey-articles that can be used as an introduction to the subject are,
for example,
\cite[Sections 1.3 and 4.2]{Mazzeo-Piazza-II}, \cite{loya-survey}, \cite{grieser-survey}.}

Let $\tM_0$ be a Galois $\Gamma$-covering of $M_0$; we let $\tilde g_0$ be the lifted metric. We also consider  $\partial \tM_0$, the
boundary of $\tM_0$.
We consider
$\tM:= \tM_0 \cup_{\partial \tM_0} \left(   (-\infty,0] \times \partial\tM_0 \right),$
endowed with the extended metric $\tilde{g}$ and the obviously extended $\Gamma$ action along the cylindrical
end.
Notice  that we obtain in this way a $\Gamma$-covering of manifolds with cylindrical ends
\begin{equation}\label{cylindrical-covering}
\Gamma\to \tM \to M
\end{equation}
With a small abuse  we  introduce the notation:
\begin{equation}\label{cyl-notation}
\cyl (\pa \tM):= \RR\times \partial \tM_0\;,\;\;\;
\cyl^- (\pa \tM):=(-\infty,0] \times \partial \tM_0
\end{equation}
and \begin{equation}\label{cyl-notation-bis}
\cyl^+ (\pa \tM):=
[0,+\infty) \times \partial \tM_0\,.
\end{equation}
(The abuse of notation is in  writing $\cyl (\pa \tM)$ for $ \RR\times \partial \tM_0$ whereas we should
really write $\cyl (\pa \tM_0)$.)

We assume the existence of  a bundle of Clifford modules $E$, endowed with a hermitian
metric $h$ for which the Clifford action is
unitary, and  equipped with a Clifford connection. We assume product structures near the boundary throughout.

\subsection{Dirac operators} $\;$\\
Associated to the above structures there is a generalized
Dirac operator $D$ on $M_0$ with product structure near the
boundary. We denote by $D_{\pa}$ the operator induced on the boundary.
We employ the same symbol for the associated $b$-Dirac operator on $M$.
We denote by $\widetilde{D}$ and $\widetilde{D}_{\partial}$ the $\Gamma$-equivariant $b$-Dirac operators on $\widetilde{M}$ and $\partial \widetilde{M}$.
We also have $D_{\cyl}$ on $\RR\times \pa M_0\equiv \cyl(\pa M)$
and  $\widetilde{D}_{\cyl}$
on $\RR\times \pa \widetilde{M}_0\equiv \cyl (\pa  \widetilde{M})$.
Next we consider $\Lambda:=C^*_r \Gamma$, the reduced group $C^*$-algebra
and  $\B^\infty\subset \Lambda$  the Connes-Moscovici algebra (we recall its definition
in Section \ref{sect:continuity}); we denote
by $\D_\Lambda$ and $\D_\infty$ the Dirac operators obtained by  twisting $D$ by the
Mishchenko bundle  $\V:=\tilde{M}\times_\Gamma \Lambda$
and the $\B^\infty$-Mishchenko
bundle $\V^\infty:=\tilde{M}\times_\Gamma \B^\infty$. Unless confusion should arise
we
denote the latter simply by $\D$.
We refer for example to \cite[Section 1]{LPGAFA} for more details
on these geometric preliminaries on Dirac operators.
We shall make the following fundamental assumption
\begin{assumption}\label{assumption:invertibility}
There exists a $\delta>0$ such that
\begin{equation}\label{invertibility}
{\rm spec}_{L^2} (\widetilde{D}_{\partial})\cap [-\delta,\delta]=\emptyset
\end{equation}
\end{assumption}
It should be noticed that because of the self-adjointness of $\widetilde{D}_{\partial}$,
assumption \eqref{invertibility} implies the $L^2$-invertibility of $\widetilde{D}_{\cyl}$.
A detailed proof of this implication is given, for example,  in \cite{LPETALE} (but in a  more general
situation) and we refer the reader to that paper for details (see page 188 there, between
(11) and (12), and then page 189); the   idea is to conjugate the operator  $\widetilde{D}_{\cyl}$,
$$\widetilde{D}_{\cyl}=\left(\begin{array}{cc} 0&-\frac{\pa}{\pa t} + \widetilde{D}_{\partial}\\
\frac{\pa}{\pa t}  + \widetilde{D}_{\partial}&0 \end{array} \right),
$$
 by Fourier transform
in $t$, $\mathcal{F}_{t\to\lambda}$, obtaining \begin{equation}\label{cyl-ind}
\left(\begin{array}{cc} 0&-i\lambda + \widetilde{D}_{\partial}\\
i\lambda + \widetilde{D}_{\partial}&0 \end{array} \right).
\end{equation}
In the $b$-calculus-picture \eqref{cyl-ind}  is the indicial family $I(\widetilde{D}_{\cyl},\lambda)$
of $\widetilde{D}_{\cyl}$ and it is obtained through Mellin transform of the corresponding cylindrical
$b$ operator.
The self-adjointness of $ \widetilde{D}_{\partial}$ implies that
\eqref{cyl-ind}, i.e. $I(\widetilde{D}_{\cyl},\lambda)$, is $L^2$-invertible
for each $\lambda\in\RR\setminus\{0\}$; the invertibility of  $\widetilde{D}_{\partial}$ then implies that
\eqref{cyl-ind} is $L^2$-invertible for each $\lambda\in\RR$. Conjugating back the inverse of \eqref{cyl-ind}
one obtains an operator which provides an $L^2$-inverse of $\widetilde{D}_{\cyl}$.
Essentially the same argument shows
the invertibility of $\D_{\cyl}$ in the
 $\B^\infty$-Mishchenko-Fomenko $b$-calculus with bounds (to be introduced in Section \ref{sect:index}).
 The details are as follows:  $\D_{\Lambda,\pa}$ is a self-adjoint regular unbounded operator on the
Hilbert $\Lambda$-module of $L^2$-sections of $(E \otimes \mathcal{V})_{\pa}$: $L^2 (\pa M_0, (E \otimes \mathcal{V})_{\pa})$. Hence,
 because of Assumption
\eqref{invertibility}, see for example \cite[Lemmas 2.1 and 3.1]{LLP} we know that  for each $\lambda\in\RR$
the operator
\begin{equation}\label{cyl-ind-Lambda}
\left(\begin{array}{cc} 0&-i\lambda + \D_{\Lambda,\pa}\\
i\lambda + \D_{\Lambda,\pa}&0 \end{array} \right).
\end{equation}
is invertible with inverse a $\Lambda$-pseudodifferential operator of order $-1$. Using the arguments
 in \cite{LPMEMOIRS} and in \cite[Appendix]{LPBSMF},
in turn based on the work of Lott \cite{Lott-torsion}, we conclude that the indicial family of
$\D_{\cyl}$, the latter being a differential operator of order 1 in the $\B^\infty$-Mishchenko-Fomenko $b$-calculus,  is invertible for each $\lambda\in \RR$, with inverse $I(\D_{\cyl},\lambda)^{-1} \in \Psi^{-1}_{\B^\infty}$ $\forall \lambda\in\RR$.
Proceeding now as in \cite{Melrose}, Sections 5.7 and 5.16, we conclude, using the inverse Mellin transform, that there exists an inverse of
$\D_{\cyl}$ and that this inverse is an element in $\B^\infty$-Mishchenko-Fomenko $b$-calculus with $\epsilon$-bounds, with $\epsilon<\delta$
and $\delta$ as in \eqref{invertibility}.

\subsection{Cyclic homology}$\;$\\
In this paper we use the periodic version of cyclic homology and cohomology. In this section we briefly recall definitions and notations we use. The general references for this material are \cite{Loday, Karoubi, lmpflaum, LMP2}.

\noindent

Let $\A$ be a complex unital algebra. Set $C_k(\A)= \A\otimes (\A/\mathbb{C} 1)^{\otimes k}$ for $k \ge 0$,
$C_k(\A)= 0$ for $k <0$. Since the algebras we consider will be Fr\'echet algebras, the tensor product is understood to be completed so that $C_k(\A)$ is a Fr\'echet space.
The space of normalized periodic cyclic chains of degree $ l \in \mathbb{Z}$ is defined by
\begin{equation*}
CC_{l}(\A)  =\prod \limits_{n \in \mathbb{Z}} C_{l+2n}(\A).
\end{equation*}
The boundary is given by $b+B$ where $b$ and $B$ are the Hochschild and Connes boundaries of the cyclic complex.
The homology of this complex is denoted $HC_{\bullet}(\A)$.

If $\A$ is not necessarily unital denote by $\A^+$ its unitalisation and set $CC_{l}(\A)=CC_{l}(\A^+)/CC_{l}(\mathbb{C})$.
For a unital $\A$ this complex is quasiisomorphic to the one previously described.

If
$I\colon \A \to \G$ is a  homomorphism of algebras, one can consider the relative cyclic complex  $CC_\bullet(\A, \G)$ which is the shifted cone of the morphism of cyclic complexes induced by $I$, see \cite{lmpflaum}. Explicitly,
\begin{equation*}
CC_k(\A, \G) = CC_k(\A) \oplus CC_{k+1}(\G),
\end{equation*}
with the differential given by
\begin{equation*}
(\alpha, \gamma) \mapsto ((b+B) \alpha, -I(\alpha) - (b+B) \gamma), \alpha \in  CC_k(\A), \gamma \in CC_{k+1}(\G).
\end{equation*}

In a dual manner we also consider the cyclic cohomology associated to $\A$. For a unital $\A$ and $k\ge 0$  $C^k(\A)$ denotes the space of continuous $k+1$ linear forms $\phi$ on $\A$ with the property that $\phi(a_0, \ldots, a_{i-1}, 1,a_i, \ldots a_{k-1})=0$, $1\le i\le k$. We set
$C^k(\A)=0$ for $k<0$.
\begin{equation*}
CC^{l}(\A)  =\bigoplus \limits_{n \in \mathbb{Z}} C^{l+2n}(\A).
\end{equation*}
and the differential is given by the (transposed of ) $b+B$.
for $k \ge 0$. There is a natural pairing $\langle \cdot, \cdot \rangle$ between $CC^{l}(\A)$ and $CC_{l}(\A)$ which induces a pairing
\begin{equation}\label{pairing-cyclic-coh}
\langle\; , \; \rangle_{HC}:  HC_\bullet(\A)\otimes HC^\bullet(\A)\to \CC
\end{equation}

If
$I\colon \A \to \G$ is a  homomorphism, the relative cohomological complex  $CC^\bullet(\A, \G)$  is given by
\begin{equation*}
CC^k(\A, \G) = CC^k(\A) \oplus CC^{k+1}(\G),
\end{equation*}
with the differential given by
\begin{equation*}
(\phi, \psi) \mapsto ((b+B) \phi -I^* \psi,  - (b+B) \psi), \phi \in  CC^k(\A), \psi \in CC^{k+1}(\G).
\end{equation*}

The pairing between  $CC_\bullet(\A, \G)$ and $CC^\bullet(\A, \G)$ is given by
\begin{equation}\label{pairing-cyclic-relative}
\langle (\alpha, \gamma), (\phi, \psi)\rangle = \langle \alpha, \phi \rangle+ \langle \gamma, \psi\rangle\,;
\end{equation}
it induces a pairing
\begin{equation}\label{pairing-cyclic-relative-coh}
\langle\; , \; \rangle_{HC}:  HC_\bullet(\A, \G)\otimes HC^\bullet(\A, \G)\to \CC
\end{equation}

Recall that for an algebra $\A$ we have  a Chern character in cyclic homology   $\ch \colon K_0(\A)\to HC_{0}(\A)$.
It is  defined by the following formula. Let $P, Q \in M_{n\times n}(\A^+)$ be two idempotents in $n\times n$ matrices of the
algebra $\A^+$ such that $P-Q \in M_{n\times n}(\A)$. Note that this pair of idempotents  represents a class $[P]-[Q] \in K_0(\A)$. Then
\begin{equation}\label{cyclicchern}
{\Ch\left(P-Q\right) = \tr(P-Q)+ \sum_{n=1}^{\infty} (-1)^n\frac{(2n)!}{n!} \tr\left( \left(P-\frac{1}{2}\right)\otimes P^{\otimes (2n)}-  \left(Q-\frac{1}{2}\right)\otimes Q^{\otimes (2n)} \right)}
\end{equation}
We will use the notation $\Ch\left(P-Q\right)$ for the cyclic cycle defined above and $\ch\left([P]-[Q]\right)$ for
its class in cyclic homology $HC_0(\A)$.

Assume  for a moment that $\A$   is a \emph{unital}   Fr\'echet algebra and $p_t$, $t \in [0,1]$ is a smooth path of idempotents in    $M_{n\times n}(\A)$.
Then
\begin{equation*}
\Ch(p_1) - \Ch(p_0) = (b+B) \Tch (p_t).
\end{equation*}
Here the components of the chain $\Tch =\sum_{n=0}^\infty \Tch_{2n+1}$ are given by
\begin{equation*}
\Tch_1(p_t)= -\int \limits_0^1 \tr (p_t\otimes [\dot{p}_t, p_t]) dt
\end{equation*}
\begin{equation*}
  \ \Tch_{2n+1}(p_t) = (-1)^n\frac{(2n)!}{n!} \int \limits_0^1 \sum \limits_{i=0}^{2n} (-1)^{i+1} \tr \left(p_t-\frac{1}{2}\right)\otimes p_t^{\otimes i}\otimes[\dot{p}_t, p_t] \otimes p_t^{\otimes (2n-i)}\, dt,
\end{equation*}
where $\dot{p}_t =\frac{dp_t}{dt}$.

If $\A$ and $\G$ are Fr\'echet algebras, unital or not,  and
$I \colon \A \to \G$ is  a  continuous homomorphism, then
an element in the relative group $K_0(\A,\G) =K_0(\A^+,\G^+)$ can be represented by a
 triple
$(e_1, e_0, p_t)$ with
$e_0$ and $e_1$  projections in  $M_{n\times n}(\A^+)$, and
$p_t$  a smooth family of projections in $M_{n\times n}(\G^+)$, $t\in [0,1]$, satisfying
$I (e_i)=p_i$ for $i=0,1$. Here $I$ denotes also  the induced homomorphism $\A^+ \to \G^+$.

There is a Chern character
$\ch \colon K_0(\A, \G) \to HC_{0}(\A, \G)$ given  by
\begin{equation*}
\ch([(e_1, e_0, p_t)]) = (\Ch (e_1 - e_0) , -\Tch(p_t)).
\end{equation*}
We will also use pairings between $K$-theory and cyclic cohomology given by
\begin{equation}\label{paring-abs}
\langle [e_1]-[e_0], [\tau] \rangle:=\langle \ch([e_1]-[e_0]), [\tau] \rangle_{HC}
\end{equation}
in the absolute case and by
\begin{equation}\label{paring-rel}
 \langle [(e_1, e_0, p_t)], [(\tau, \sigma)] \rangle:= \langle \ch ([(e_1, e_0, p_t)]), [(\tau, \sigma)] \rangle_{HC}.
\end{equation}
in the relative case.

\subsection{Noncommutative de Rham homology}\label{subsect:karoubi}$\;$\\
For a unital algebra $A$ let $\Omega^\bullet A$ be the free unital differential graded algebra generated by $A$. The differential
in $\Omega^\bullet A$ is denoted by $d$.
$\Omega^k A$  is the span of expressions of the form $a_0da_1 \ldots da_k$. Set $\overline{\Omega}^\bullet A =\Omega^\bullet A / [\Omega^\bullet A, \Omega^\bullet A]$. $d$ defines a map $\overline{\Omega}^\bullet A \to \overline{\Omega}^{\bullet+1} A$.
 Then Karoubi's homology  $\overline{H}_\bullet (A)$ is the cohomology of the complex  $\left( \overline{\Omega}^\bullet A ,d \right)$.
The Chern character $K_0(A) \to \prod_i \overline{H}_{2i} (A)$ is defined as follows. Let $e \in M_{n\times n}(A)$ be an idempotent.
Then Chern character of $[e]$ is represented by the form
\begin{equation*}
\Ch_K(e) = \trace \exp (-e de de).
\end{equation*}
Consider now the reduced cyclic complex $\overline{C}_{\lambda}^\bullet(A)$. In degree $\ell$ it consists of $(\ell+1)$-linear functionals $\phi$ on
$A$ satisfying $\phi(a_\ell, a_0, a_1, \ldots, a_{\ell-1}) =(-1)^\ell \phi(a_0, a_1, \ldots, a_{\ell-1}, a_\ell)$, $\phi(1, a_1, \ldots a_{\ell-1})=0$.
The differential is given by $b$. $\overline{C}_{\lambda}^\bullet(A)$ is naturally a subcomplex of $CC^\bullet(A)$, as $B$ vanishes on
$\overline{C}_{\lambda}^\bullet(A)$. The cohomology of $\overline{C}_{\lambda}^\bullet(A)$ is denoted $\overline{H}_{\lambda}^\bullet(A)$. By the above discussion there is a natural map $\iota \colon \overline{H}_{\lambda}^\bullet(A) \to HC^\bullet(A)$.

There is a natural pairing between $\overline{H}_\bullet (A)$ and $\overline{H}_{\lambda}^\bullet(A)$ given by
\begin{equation*}
\langle \sum a_0da_1\ldots da_\ell,  [\tau] \rangle_K := \ell! \sum \tau(a_0, a_1, \ldots, a_\ell).
\end{equation*}
Then for $[\tau] \in \overline{H}_{\lambda}^\bullet(A)$ we have
\begin{equation}\label{pairing-compatibily}
\langle \Ch_k(e), [\tau] \rangle_K = \langle [e], \iota [\tau] \rangle:= \langle \ch[e],\iota[\tau]
\rangle_{HC}.
\end{equation}
\subsection{Group cohomology}
Let $\Gamma$ be a discrete group. The homogeneous complex $C_{hom}^\bullet(\Gamma, \CC)$ computing the cohomology of $\Gamma$  can be described as follows:
\begin{equation*} C_{hom}^k(\Gamma, \CC) = \{\phi \colon \Gamma^{k+1} \to \CC\ |\ \phi(gg_0, \ldots gg_k) = \phi(g_0, \ldots, g_k)\}.\end{equation*}
The differential is given by
\begin{equation*}\partial \phi(g_0, \ldots, g_k) = \sum_i (-1)^i\phi(g_0,\ldots, g_{i-1}, g_{i+1}\ldots g_k) \end{equation*}
This complex is isomorphic to the nonhomogeneous complex
\begin{equation*} C_{nhom}^k(\Gamma, \CC) = \{c \colon \Gamma^{k} \to \CC\}\end{equation*}
with the differential
\begin{equation*}
\delta c( g_1, \ldots, g_{k+1}) = c(g_2, \ldots, g_{k+1}) + \sum (-1)^i c(g_1, \ldots, g_i g_{i+1} \ldots, g_{k+1}) +(-1)^{k+1}  c( g_1, \ldots, g_{k}).
\end{equation*}

The isomorphism of complexes is given by $I \colon C_{nhom}^\bullet(\Gamma, \CC) \to C_{hom}^\bullet(\Gamma, \CC)$:
\begin{equation*}
I(c)(g_0, \ldots, g_k) = c(g_0^{-1}g_1, g_1^{-1}g_2, \ldots, g_{k-1}^{-1}g_k)
\end{equation*}
with the inverse map
\begin{equation*}
I^{-1}(\phi)(g_1, \ldots, g_k) = \phi(1, g_1, g_1 g_2, \ldots, g_1\ldots g_k)
\end{equation*}

One can consider the subcomplex $C_{hom, \Lambda}^\bullet(\Gamma, \CC)  \subset C_{hom}^\bullet(\Gamma, \CC)$ defined by
\begin{equation*}
C_{hom, \Lambda}^\bullet(\Gamma, \CC)= \{\phi \in  \subset C_{hom}^\bullet(\Gamma, \CC) \ |\ \phi(g_{\sigma(0)}, g_{\sigma(1)}, \ldots, g_{\sigma(k)}) = \sgn \sigma \phi (g_0, \ldots, g_k) \text{ for every }  \sigma \in S_{k+1} \}.
\end{equation*}
The inclusion $C_{hom, \Lambda}^\bullet(\Gamma, \CC)  \subset C_{hom}^\bullet(\Gamma, \CC)$ is a quasiisomorphism.

In this paper we will be working with the complex
\begin{equation}\label{groupcoc} C^\bullet(\Gamma, \CC) \subset C_{nhom}^\bullet(\Gamma, \CC)\end{equation}
 which is the image of $C_{hom, \Lambda}^\bullet(\Gamma, \CC)$ under the map $I^{-1}$. $Z^\bullet(\Gamma, \CC) =\Ker\left( \delta \colon  C^\bullet(\Gamma, \CC)\to C^{\bullet+1}(\Gamma, \CC)\right)$ denotes the subspace of group cocycles.
We note several immediate properties of the cochains in $C^\bullet(\Gamma, \CC)$:
\begin{lemma}
\label{lemma:group-cohom} Let $c \in C^\bullet(\Gamma, \CC)$.
\begin{enumerate}
\item  $c$ is normalised, i.e. $c(g_1,\ldots, g_k)=0$ if $g_i=1$ for some $i$ or $g_1g_2\ldots g_k=1$.
\item Let $g_{ij} \in \Gamma$, $i$, $j=0,1,\ldots, m$ be such that $g_{ij} g_{jk}=g_{ik}$ for every $i$, $j$, $k$.
Then the expression $c(g_{i_0 i_1}, g_{i_1 i_2}, \ldots g_{i_{k-1} i_k})$ is antisymmetric in $i_0$, $i_1$,\ldots, $i_k$.
\item \label{c} If $g_1 g_2 \ldots g_{k+1}=1$, then
\begin{equation*}
c(g_2,\dots,g_{k+1})= (-1)^{k} c (g_1, \dots, g_{k}).
\end{equation*}
\end{enumerate}
\end{lemma}

\section{Index classes}\label{sect:index}

\subsection{The index class $\Ind_\infty (\D)$}\label{subsect:index} $\;$\\
 Let $\epsilon>0$ be strictly smaller than $\delta$, the width of the spectral gap for the boundary operator appearing in \ref{assumption:invertibility}.
 We introduce
 \begin{itemize}
 \item
 $A:=\Psi^{-\infty,\epsilon}_b (M,E) + \Psi^{-\infty,\epsilon} (M,E)$,
 the sum of the smoothing
operators in the  $b$-calculus with $\epsilon$-bounds and the residual
operators in the $b$-calculus with $\epsilon$-bounds.
\item $J:= \Psi^{-\infty,\epsilon} (M,E)$
\end{itemize}
We know that $A$ is an algebra
and  that $J$ is an ideal in $A$ (see \cite{Melrose}
and, for this particular result,  \cite[Theorem 4]{MPI}).

 \begin{itemize}
\item $G:=  \Psi^{-\infty,\epsilon}_{b,\RR^+} (\overline{N^+ (\pa M)},E) +
\Psi^{-\infty,\epsilon}_{\RR^+}  (\overline{N^+ (\pa M)},E)$, the $\RR^+$-invariant smoothing
operators in the $b$-calculus with $\epsilon$-bounds on the compactified
positive normal bundle to the boundary. Here we have abused notation and used $E$
to denote the extension to the normal bundle of the restriction of $E$ to the boundary.
\end{itemize}
We know, see \cite{Melrose}, that there is a short exact sequence of algebras
$$0\rightarrow J \rightarrow A\xrightarrow{I} G
\rightarrow 0$$
with $I$ denoting the map equal to the {\em indicial operator}
on $\Psi^{-\infty,\epsilon}_b (M,E)$ and equal to zero on $ \Psi^{-\infty,\epsilon} (M,E)$.

\smallskip
We then consider (we write MF for Mishchenko-Fomenko):
\begin{itemize}
\item  the algebra $\mathfrak{A}:= \Psi^{-\infty,\epsilon}_b (M,E\otimes
\V^\infty) + \Psi^{-\infty,\epsilon} (M,E\otimes \V^\infty)$,  the  sum of the smoothing
operators in the $\B^\infty$-MF $b$-calculus with $\epsilon$-bounds and the residual
operators in the $\B^\infty$-MF $b$-calculus with $\epsilon$-bounds;
\item the ideal  $\mathfrak{J}$ in $\mathfrak{A}$ equal to the residual
operators $\Psi^{-\infty,\epsilon} (M,E\otimes \V^\infty)$;
\item  the algebra $\mathfrak{G}$ of $\RR^+$-invariant smoothing
operators in the $\B^\infty$-MF $b$-calculus with $\epsilon$-bounds on the compactified
positive normal bundle to the boundary
\end{itemize}
Considering the map $I: \mathfrak{A}\to \mathfrak{G}$
equal to zero on the residual operators
and equal to the indicial operator on the smoothing
operators in the $\B^\infty$-MF $b$-calculus with bounds, we get
 a short exact sequence
$$0\rightarrow \mathfrak{J} \rightarrow \mathfrak{A}\xrightarrow{I} \mathfrak{G}
\rightarrow 0$$

One can prove, see \cite{LPMEMOIRS} and \cite[Appendix]{LPBSMF}, that $\D^+$ is invertible modulo elements in $\mathfrak{J}$; if $\Q$ is a parametrix with remainders $\mathcal{S}_\pm$ then we can consider
 the Connes-Skandalis projector
\begin{equation}\label{CS-projector}
P_{\Q}:= \left(\begin{array}{cc} \mathcal{S}_{+}^2 & \mathcal{S}_{+}  (I+\mathcal{S}_{+}) \Q\\ \mathcal{S}_{-}\D^+ &
I-\mathcal{S}_{-}^2 \end{array} \right).
\end{equation}
We obtain in this way an index class
\begin{equation}\label{CS-class}
\operatorname{CS}_\infty (\D):= [P_{\Q}] - [e_1]\in K_0 (\mathfrak{J})\;\;\text{with}\;\;e_1:=\left( \begin{array}{cc} 0 & 0 \\ 0&1
\end{array} \right)
\end{equation}
see, for example,  \cite{Co} (II.9.$\alpha$) and  \cite{CM} (p. 353).

We denote
by $\operatorname{CS}_{\Lambda} (\D)$ (recall that $\Lambda=C^*_r \Gamma$)
 the image of this class in $K_0 (C^*_r\Gamma)$
through the homomorphism $\iota_*$ associated to the natural inclusion $\iota: \mathfrak{J}\to \mathbb{K}(\mathcal{E}_{\rm MF})$;
here $\mathcal{E}_{\rm MF}$ is the $\Lambda$-Hilbert module given by
$L^2 (M,E\otimes \mathcal{V})$. It is clear that this is the Connes-Skandalis class
associated to a parametrix for $\D_{\Lambda}$. We shall often denote this class simply by
$\operatorname{CS} (\D)$; thus
\begin{equation}\label{c*-cs}
\operatorname{CS} (\D):= \iota_* (\operatorname{CS}_\infty (\D))
\end{equation}

As in Connes-Moscovici \cite[Section 5]{CM}, we can use a trivializing open cover
of $M_0$, with $k$ trivializing open sets,
a partition of
unity associated to it and a collection of local sections in order
to define an isometric embedding
\begin{equation}\label{isometric-u}
C^\infty (M,E\otimes \V^\infty)\xrightarrow{U}
C^\infty (M, E\otimes( \B^\infty\otimes \CC^k))
\end{equation}
with the trivializing open cover extended to $M$ in the obvious way.
Then $\theta (A):= U A U^*$ defines an algebra homomorphism
between the algebra $\mathfrak{A}$, i.e. $\Psi^{-\infty,\epsilon}_b (M,E\otimes
\V^\infty) + \Psi^{-\infty,\epsilon} (M,E\otimes \V^\infty)$, and the algebra $\A$ defined by
\begin{equation}\label{calligA}
\A:=
\Psi^{-\infty,\epsilon}_b (M,E\otimes (\B^\infty\otimes \CC^k)) + \Psi^{-\infty,\epsilon} (M,E\otimes (\B^\infty\otimes \CC^k))
\end{equation}
and obtained by considering the relevant MF-calculi with values in the
trivial  bundle $\B^\infty\otimes \CC^k$. \\
We obtain also
$$\J:= \Psi^{-\infty,\epsilon} (M,E\otimes (\B^\infty\otimes \CC^k))\quad\text{and}\quad
\G:= \Psi^{-\infty,\epsilon}_{b,\RR^+} (\overline{N^+ (\pa M)},E\otimes (\B^\infty\otimes \CC^k)) +
\Psi^{-\infty,\epsilon}_{\RR^+}  (\overline{N^+ (\pa M)},E\otimes (\B^\infty\otimes \CC^k))$$
and we know that there is a short exact sequence of
algebras
$$0\to \J\to \A\xrightarrow{I}  \G\to 0$$
We can similarly define a homomorphism
$\theta_{\cyl}: \mathfrak{G} \to \G$ and a simple argument with coverings shows that there exists a commutative
diagram
\begin{equation}\label{compatibility-theta}
\xymatrix{
0\ar[r]\ar[d] & \mathfrak{J}\ar[r]\ar[d]^{\theta} &\mathfrak{A} \ar[r]^{I} \ar[d]^{\theta}
& \mathfrak{G} \ar[r]\ar[d]^{\theta_{\cyl}} & 0\ar[d]\\
0\ar[r] & \mathcal{J}\ar[r] &\mathcal{A} \ar[r]^{I}
& \mathcal{G} \ar[r] & 0}
\end{equation}


Let $\Theta: K_0 (\mathfrak{J}) \to K_0 (\mathcal{J})$ be the homomorphism
defined by $\theta$; as explained in \cite{CM} this homomorphism is well-defined,
independent of the choices we have made in its definition.

\begin{definition}\label{def:index-class}
 Inspired directly by  \cite{CM}
we define
\begin{equation}\label{eq:index-class}
\Ind_\infty (\D):= \Theta (\operatorname{CS}_{\infty} (\D))\in K_0 (\J)
\end{equation}
where we recall that $\J=
\Psi^{-\infty,\epsilon} (M,E\otimes (\B^\infty\otimes \CC^k))$.
We can also define $\Ind (\D):= \iota_* (\Ind_\infty (\D)) \in K_0 (C^*_r \Gamma)$
with $\iota$ equal to the composition of inclusions
$$\Psi^{-\infty,\epsilon} (M,E\otimes(\B^\infty\otimes \CC^k))\to
\Psi^{-\infty,\epsilon} (M,E\otimes (C^*_r \Gamma\otimes \CC^k))
\to \mathbb{K}(\mathcal{E}_{{\rm MF}}^{\otimes})$$
with $\mathcal{E}_{{\rm MF}}^{\otimes}$ equal to the
$C^*_r \Gamma$-Hilbert module $L^2 (M, E\otimes (C^*_r \Gamma\otimes \CC^k))$.
\end{definition}

\subsection{The Connes-Moscovici idempotent}\label{subsection:other-index-classes}$\;$\\
There are  descriptions of the class $\operatorname{CS}_\infty (\D)\in
K_0 (\mathfrak{J})$, and thus of the index class $\Ind(\D)_{\infty}\in K_0 (\mathcal{J})$,
that are particularly useful in computations.

\noindent
First, for motivation, consider
a closed compact manifold $N$ and a Galois $\Gamma$-covering $\widetilde{N}$.
Consider  $\operatorname{CS} (\D)\in
K_0 (C^*_r \Gamma)$.
We can make  two specific
choices of parametrices for $\D^+$ in \eqref{CS-projector}:
\begin{equation}\label{qe+qw}
\Q_{e} :=  (I+ \D^- \D^+)^{-1} \D^-\quad\text{and}\quad \Q_{V}=
:= \frac{I-\exp(-\frac{1}{2} \D^- \D^+)}{\D^- \D^+} \D^+
\end{equation}
with $I- \Q_{e}\, \D^+ = (I+\D^- \D^+)^{-1}$, $ I-\D^+ \, \Q_{e}= (I+\D^+ \D^-)^{-1}$ and
$I-\Q_{V}\D^+ = \exp(-\frac{1}{2} \D^- \D^+)$, $I-\D^+ \Q_{V} =  \exp(-\frac{1}{2} \D^+ \D^-)$.\\
The first choice of parametrix produces
the graph projection
\begin{equation}\label{graph}
e_{\D}=\left(\begin{array}{cc} (I+\D^- \D^+)^{-1} & (I+\D^- \D^+)^{-1} \D^-\\ \D^+   (I+\D^- \D^+)^{-1} &
\D^+  (I+\D^- \D^+)^{-1} \D^- \end{array} \right).
\end{equation}
The choice of $\Q_{V}$ produces the
idempotent
\[V_{\D}=\left( \begin{array}{cc} e^{-\D^- \D^+} & e^{-\frac{1}{2}\D^- \D^+}
\left( \frac{I- e^{-\D^- \D^+}}{\D^- \D^+} \right) \D^-\\
e^{-\frac{1}{2}\D^+ \D^-}\D^+& I- e^{-\D^+ \D^-}
\end{array} \right)
\]
We call this the Connes-Moscovici idempotent (see \cite{CM}).

One can also consider
  the Wassermann projection $W_{\D}$,
\begin{equation}\label{wass}
W_{\D} := \left( \begin{array}{cc} e^{-\D^- \D^+} & e^{-\frac{1}{2}\D^- \D^+}
\left( \frac{I- e^{-\D^- \D^+}}{\D^- \D^+} \right)^{ \frac{1}{2}} \D^-\\
e^{-\frac{1}{2}\D^+ \D^-}\left( \frac{I- e^{-\D^+ \D^-}}{\D^+ \D^-} \right)^{ \frac{1}{2}} \D^+& I- e^{-\D^+ \D^-}
\end{array} \right)\,,
\end{equation}
homotopic to $V_{\D}$ via
\begin{equation}\label{CM-homotopy}
P_{\D} (s):= \left( \begin{array}{cc} e^{-\D^- \D^+} & e^{-\frac{1}{2}\D^- \D^+}
\left( \frac{I- e^{-\D^- \D^+}}{\D^- \D^+} \right)^{ \frac{1}{2}+s} \D^-\\
e^{-\frac{1}{2}\D^+ \D^-}\left( \frac{I- e^{-\D^+ \D^-}}{\D^+ \D^-} \right)^{ \frac{1}{2}-s} \D^+& I- e^{-\D^+ \D^-}
\end{array} \right)\,,
\end{equation}
with $s\in [0,1/2]$. See \cite{CM},
before Lemma (2.5).

The same formula \eqref{CM-homotopy} with $s\in [-1/2,1/2]$ defines a homotopy between
$V_{\D}$ and $V_{\D}^*$.  From the discussion above we obtain the following equality of elements in
$K_0 (C^*_r \Gamma)$:

\begin{equation}\label{classeq}
[P_{\Q}] -[e_1]
=[e_{\D}] -[e_1]= [V_{\D}]-[e_1]
= [V_{\D}^*]-[e_1]
\end{equation}
We will also need to consider  a symmetrized idempotent $\dvd:= V_{\D} \oplus V_{\D}^*$.
If we set  $\del:=e_1 \oplus e_1$, we have
\[
[\dvd] -[\del]=2([V_{\D}]-[e_1]).
\]

The well-known properties of the heat operator and of the pseudodifferential calculus  imply that  actually
  $[V_{\D}]-[e_1]$ and $[\dvd]-[\del]$ belong to $ K_0 (\Psi^{-\infty}(N,E\otimes \mathcal{V}^\infty))$.

Summarizing, we can set
\begin{equation}
\operatorname{CS}_\infty (\D):=[V_{\D}]-[e_1]= [V_{\D}^*]-[e_1]
\;\;\text{in}\;\;\,K_0 (\Psi^{-\infty}(N,E\otimes \mathcal{V}^\infty)).
\end{equation}
The last equality follows from the fact that $\mathcal{B}^\infty$ is closed under holomorphic functional calculus
in $C^*_r \Gamma$. Indeed, the images of the classes $[V_{\D}]-[e_1]$ and $[V_{\D}^*]-[e_1]$ are equal in $K_0 (C^*_r \Gamma)$,
see \eqref{classeq}, and the map
\[K_0 (\Psi^{-\infty}(N,E\otimes \mathcal{V}^\infty)) \to K_0 (C^*_r \Gamma)\]
is an isomorphism.

Consider now the isometric embedding
$C^\infty (M,E\otimes \V^\infty)\xrightarrow{U}
C^\infty (M, E\otimes( \B^\infty\otimes \CC^k))$ recalled in \eqref{isometric-u}; one can check  that
$U\D U^*= \D^\otimes$, with $\D^\otimes$
equal to the operator $D$ twisted by the trivial bundle $\B^\infty\otimes \CC^r$.
This implies that $\theta (V_{\D})= V_{\D^\otimes}$ and  $\theta (V_{\D}^*)= V_{\D^\otimes}^*$.
We obtain immediately:
\begin{equation}\label{theta-of-w}
\Ind_\infty (\D):=
\Theta (\operatorname{CS}_\infty (\D))=
[V_{\D^\otimes}]-[e_1]=[V_{\D^\otimes}^*]-[e_1]\;\;\text{in}\;\;K_0 (\Psi^{-\infty}(N,E\otimes (\mathcal{B}^\infty\otimes\CC^r)))
\end{equation}

Let us now pass to a $b$-manifold $M$ and to an operator $\D$ satisfying
the invertibility assumption on the boundary.
First of all,  recall how the (true) parametrix of $\D^+$ is constructed.
We shall be somewhat brief on this point since this procedure is explained in detail
in many places; in particular we shall not be particularly precise about the
gradings and the
identifications on the boundary.
One begins
by finding a symbolic parametrix $\Q_\sigma$ to $\D^+$, with remainders $\R^\pm_\sigma$.
Next, by fixing a cut-off function $\chi$ on the collar neighborhood of the
boundary, equal to 1 on the boundary, we define a section $s:\mathfrak{G} \to \mathfrak{A}$ to the
indicial homomorphism $I: \mathfrak{A}\to \mathfrak{G}$. $s$ associates to a translation invariant operator
$G$ on the cylinder an operator on the manifold with cylindrical end; the latter is obtained by pre-multiplying and post-multiplying
$G$
by the cut-off function $\chi$. The (true) parametrix
of $\D^+$ is defined as  $\Q^b=\Q_\sigma- \Q^\prime$
with $Q^\prime$ equal to
$s (( I(D^+)^{-1} I( \R^-_\sigma))$. Then, with this definition, one can check, using the $b$-calculus, that $\D^+ \Q^b=I-\mathcal{S}^-$ and
 $\Q^b \D^+=I-\mathcal{S}^+$
with $\mathcal{S}^\pm$  residual operators.
Now, going back to the  classes $\operatorname{CS} (\D)$
and $\operatorname{CS}_\infty (\D)$ it is clear that
we can define the Connes-Skandalis projection using the (true) parametrix obtained through the above procedure but
starting with the symbolic parametrices $\Q_e$ and $\Q_{V}$ appearing in
\eqref{qe+qw}.
Recall that the remainders $\R^\pm_\sigma$ of these two symbolic parametrices are given by the following two equations
$I- \Q_{e}\, \D^+ = (I+\D^- \D^+)^{-1}$, $ I-\D^+ \, \Q_{e}= (I+\D^+ \D^-)^{-1}$ and
$I-\Q_{V}\D^+ = \exp(-\frac{1}{2} \D^- \D^+)$, $I-\D^+ \Q_{V} =  \exp(-\frac{1}{2} \D^+ \D^-)$.\\
The construction just explained  produces then two different (true) parametrices
$\Q^b_e$ and $\Q^b_{V}$ and two different
projectors that we denote respectively $e^b_{\D}$ and
$V^b_{\D}$. Let us see the specific structure of these two projectors, starting with $e_{\D}^b$.
 Recall that   $I(\D^\pm)=\D^\pm_{\cyl}=\pm x\pa_x + \D_{\pa}$.
By definition
\begin{equation}
Q_{e}^\prime := - \chi ((\D_{\cyl}^+)^{-1}  (I+\D^+_{\cyl} \D^-_{\cyl})^{-1} ) \chi \,.
\end{equation}
Then, a simple computation gives 
\begin{align}
\Q_{e}^\prime \D^+ &= -\chi (I+\D^-_{\cyl} \D^+_{\cyl})^{-1}\chi  + \chi (\D^+_{\cyl})^{-1}  (I+\D^+_{\cyl} \D^-_{\cyl})^{-1}
\mathop{cl}(d\chi)\\  \D^+ \Q_{e}^\prime &= -\chi (I+\D^+_{\cyl} \D^-_{\cyl})^{-1}\chi - \mathop{cl}(d\chi) (\D^+_{\cyl})^{-1}
(I+\D^+_{\cyl} \D^-_{\cyl})^{-1}\chi\,.
\end{align}
This means that
$\mathcal{S}^+_{e}:= I-\Q^b_{e} \D^+=I- (\Q_{e}-\Q_{e}^\prime)\, D^+=(I+\D^- \D^+)^{-1} + \Q_{e}^\prime \D^+$, which we know from the $b$-calculus to be residual, is given by
\begin{equation}\label{s+e}
 \mathcal{S}^+_{e}  =
   (I+\D^- \D^+)^{-1}  -\chi (I+\D^-_{\cyl} \D^+_{\cyl})^{-1}\chi
   + \chi (\D^+_{\cyl})^{-1}  (I+\D^+_{\cyl} \D^-_{\cyl})^{-1}
\mathop{cl}(d\chi).
   \end{equation}
   Similarly
\begin{equation}\label{s-e}
\mathcal{S}^-_{e}=(I+\D^+ \D^-)^{-1}- \chi (I+\D^+_{\cyl} \D^-_{\cyl})^{-1}\chi - \mathop{cl}(d\chi) (\D^+_{\cyl})^{-1}  (I+\D^+_{\cyl} \D^-_{\cyl})^{-1}\chi
 \end{equation}
Substituting $\mathcal{S}^\pm_e$ and $\Q^b_{e}$ at the place of
$\mathcal{S}^\pm$ and $\Q$ into the expression of the Connes-Skandalis projection
\begin{equation*} \left(\begin{array}{cc} \mathcal{S}_{+}^2 & \mathcal{S}_{+}  (I+\mathcal{S}_{+}) \Q\\ \mathcal{S}_{-}\D^+ &
I-\mathcal{S}_{-}^2 \end{array} \right).
\end{equation*}
we obtain $e^b_{\D}$. The precise form of  $e^b_{\D}$  plays a role in the proof of the excision
correspondence \eqref{excision-graph} explained in Theorem \ref{theo:excision} below.

Similarly, by definition,
\begin{equation}
Q_{V}^\prime := - \chi ((\D_{\cyl}^+)^{-1}  \exp(-\frac{1}{2} \D^+_{\cyl} \D^-_{\cyl}) \chi \,.
\end{equation}
and this gives us
\begin{align}
\Q_{V}^\prime \D^+ &= -\chi  \exp(-\frac{1}{2} \D^-_{\cyl} \D^+_{\cyl}) \chi  + \chi (D^+_{\cyl})^{-1}   \exp(-\frac{1}{2} \D^+_{\cyl} \D^-_{\cyl})
\mathop{cl}(d\chi)\\  D^+ \Q_{W}^\prime &= -\chi  \exp(-\frac{1}{2} \D^+_{\cyl} \D^-_{\cyl})^{-1} )\chi - \mathop{cl}(d\chi)
(D^+_{\cyl})^{-1}  \exp(-\frac{1}{2} \D^+_{\cyl} \D^-_{\cyl})\chi\,.
\end{align}

This means that
$\mathcal{S}^+_{V}:= I-\Q^b_{V}\D^+=I- (\Q_{V}-\Q_{V}^\prime)\, D^+=\exp(-\frac{1}{2} \D^+ \D^- ) + \Q_{V}^\prime D^+ $, a  residual operator, is
given by
\begin{equation}\label{s+w}
\mathcal{S}^+_{V} = \exp(-\frac{1}{2} \D^+ \D^- )
     -\chi \exp(-\frac{1}{2} \D^+_{\cyl} \D^-_{\cyl})\chi
   + \chi (D^+_{\cyl})^{-1}   \exp(-\frac{1}{2} \D^+_{\cyl} \D^-_{\cyl})
\mathop{cl}(d\chi).
   \end{equation}
It is important to remark that this is precisely the expression in \eqref{s+e} once we substitute $ \exp(-\frac{1}{2} \D^+_{\cyl} \D^-_{\cyl})$
   for $(I+\D^+ \D^-)^{-1}$; this will play a role in the excision argument to be given below.
    A similar expression can be found for $\mathcal{S}^-_{V}$.
Substituting $\mathcal{S}^\pm_V$ and $\Q^b_{V}$ at the place of
$\mathcal{S}^\pm$ and $\Q$ into the expression of the Connes-Skandalis projection
we obtain $V^b_{\D}$.

One gets as before,
$[e^b_{\D}]-[e_1]= [V^b_{\D}]-[e_1]=[\left(V^b_{\D}\right)^*]-[e_1]\in K_0 (C^*_r \Gamma)\,.$
Thus we can set:
$\operatorname{CS} (\D):=[V^b_{\D}]-[e_1] \equiv [e^b_{\D}]-[e_1] $
in
$K_0 (C^*_r \Gamma).$
One can check, using the MF-$b$-calculus with bounds, that the $b$-Connes-Moscovici
projection $V^b_{\D}$
belongs to $\mathfrak{J}^+$; crucial, here, is the information that $\D_{\pa}$
is invertible in the $\B^\infty$-MF-calculus.
(The graph projector, on the other hand, belongs to $ \Psi^{-\infty,\epsilon} (M,E\otimes \V^\infty)^+$.)
We shall choose the incarnation of the class $\operatorname{CS}_\infty (\D)$ given by $[V^b_{\D}]-[e_1]$; put it differently
\begin{equation}
\operatorname{CS}_\infty (\D):=[V^b_{\D}]-[e_1] \;\;\text{in}\;\;\,K_0 (\mathfrak{J}).
\end{equation}

Notice, finally, that from the invertibility of $\D_{\pa}$ follows the invertibility
of the boundary operator of $\D^{\otimes}$ (this is a simple consequence of
$U^* U = \Id$); proceeding as in \eqref{theta-of-w} we  obtain
\begin{equation}\label{theta-of-bw}
\Ind_{\infty} (\D):=   \Theta\left( [V^b_{\D}]-[e_1]\right) = [V_{\D^\otimes}^*]-[e_1] \;\;\text{in}\;\;\,K_0 (\mathcal{J}).
\end{equation}

\subsection{The relative index class $\Ind_\infty (\D,\D_{\pa})$. Excision.}\label{subsect:excision} $\;$\\
  Let $0\to J\to A\xrightarrow{\pi} B\to 0$ a short exact sequence of Fr\'echet
  algebras. Recall that $K_0 (J):=
K_0 (J^+,J)\cong \Ker (K_0 (J^+)\to \ZZ)$
and that $K_0 (A^+,B^+)= K_0 (A,B)$.
For the definition of relative K-groups we refer, for example, to \cite{Bla}, \cite{hr-book}, \cite{lmpflaum}.
Recall that a relative $K_0$-element
for $ A\xrightarrow{\pi} B$ with unital algebras $ A, B$
is represented by a  triple $(P,Q,p_t)$ with $P$ and $Q$ idempotents in $M_{n\times n} (A)$
and $p_t\in M_{n\times n} (B)$ a
path of idempotents connecting $\pi (P)$ to $\pi (Q)$.
The excision  isomorphism
\begin{equation}\label{excision-general}
\alpha_{{\rm ex}}: K_0 (J)\longrightarrow K_0 (A,B)
\end{equation}
is given by
$\alpha_{{\rm ex}}([(P,Q)])=[(P,Q,{\bf c})]$
with ${\bf c}$ denoting the constant path (this is not necessarily the 0-path, given that
we are taking $J^+$).

In this paper we are interested in the relative groups $K_0 (\mathfrak{A},\mathfrak{G})$
and $K_0 (\mathcal{A},\mathcal{G})$ associated respectively to
$0\to \mathfrak{J}\to\mathfrak{A}\xrightarrow{I}\mathfrak{G}\to 0$ and
$0\to \mathcal{J}\to\mathcal{A}\xrightarrow{I}\mathcal{G}\to 0$.\\
Consider the Connes-Moscovici projections $V_{\D}$ and $V_{\D_{\cyl}}$
associated to $\D$ and $\D_{\cyl}$.
With $e_1:=\begin{pmatrix} 0&0\\0&1 \end{pmatrix}$ consider the  triple
\begin{equation}\label{pre-wassermann-triple}
(V_{\D}, e_1, V_{(t\D_{\cyl})})
\,, \;\;t\in [1,+\infty]\,,\;\;\text{ with }
q_t:= \begin{cases} V_{(t\D_{\cyl})}
\;\;\quad\text{if}
\;\;\;t\in [1,+\infty)\\
e_1 \;\;\;\;\;\;\;\;\;\;\;\;\;\,\text{ if }
\;\;t=\infty
 \end{cases}
\end{equation}
Similarly, we can
  consider the triple
\begin{equation}\label{wassermann-triple}
(V_{\D}^*, e_1, V_{(t\D_{\cyl})}^*)
\,, \;\;t\in [1,+\infty].
\end{equation}

\begin{proposition}\label{prop:relative-indeces}
Under assumption \eqref{invertibility} the Connes-Moscovici idempotents $V_{\D}$
and  $V_{\D_{\cyl}}$ define  by the formula
  \eqref{pre-wassermann-triple},
a relative class in $K_0 (\mathfrak{A},\mathfrak{G})$
associated with  the short exact sequence $0\to \mathfrak{J}\to\mathfrak{A}\xrightarrow{I}\mathfrak{G}\to 0$.
We  denote this class by $[V_{\D}, e_1, V_{(t\D_{\cyl})}]$.
Similarly  the adjoint idempotents define
  a relative class  $[V_{\D}^*, e_1, V_{(t\D_{\cyl})}^*]\in K_0 (\mathfrak{A},\mathfrak{G})$.  These two classes are equal.

\end{proposition}

\smallskip
\begin{proof}
The existence of the two classes  follows from the invertibility assumption and well-known properties of the pseudodifferential calculus.
Their equality follows from the homotopy \eqref{CM-homotopy} and the fact that $\mathcal{B}^\infty$ is closed under the holomorphic functional calculus.
 \end{proof}
We set
\begin{equation}
\operatorname{CS}_\infty (\D,\D_{\pa}):= [V_{\D}, e_1, V_{(t\D_{\cyl})}] \in K_0 (\mathfrak{A},\mathfrak{G})
\end{equation}

\noindent
The homomorphisms $\theta$ and $\theta_{\cyl}$ define through \eqref{compatibility-theta}
a homomorphism $$\Theta_{{\rm rel}}: K_0 (\mathfrak{A},\mathfrak{G})\to K_0 (\mathcal{A},\mathcal{G})$$ which is well defined, independent of choices; we set
\begin{equation}\label{relative-index}
\Ind_{\infty} (\D,\D_{\pa}):= \Theta_{{\rm rel}}(\operatorname{CS}_\infty (\D,\D_{\pa}))\in
K_0 (\mathcal{A},\mathcal{G})\,.
\end{equation}
Notice that, as in \eqref{theta-of-bw},
\begin{equation}\label{was-otimes}
\Ind_{\infty} (\D,\D_{\pa})=[V_{\D^\otimes}, e_1, V_{(t\D^\otimes_{\cyl})}]= [V_{\D^\otimes}^*, e_1, V_{(t\D^\otimes_{\cyl})}^*]\,.
\end{equation}

\begin{theorem}\label{theo:excision}
Let $\alpha_{{\rm ex}}: K_0 (\mathfrak{J})\to K_0 (\mathfrak{A},\mathfrak{G})$ be the
excision isomorphism for the short exact sequence
$0\to \mathfrak{J}\to\mathfrak{A}\xrightarrow{I}\mathfrak{G}$. Then
\begin{equation}\label{excision-for-cs}
\alpha_{{\rm ex}}(\operatorname{CS}_\infty (\D))=
\operatorname{CS}_\infty (\D,\D_{\pa})\in K_0 (\mathfrak{A},\mathfrak{G}).
\end{equation}
Consequently, if $\beta_{{\rm ex}}: K_0 (\mathcal{J})\to K_0 (\mathcal{A},\mathcal{G})$
is the
excision isomorphism for the short exact sequence
$0\to \mathcal{J}\to\mathcal{A}\xrightarrow{I}\mathcal{G}$. Then
\begin{equation}\label{excision-for-ind}
\beta_{{\rm ex}}(\Ind_{\infty} (\D))=
\Ind_\infty (\D,\D_{\pa})\in K_0 (\mathcal{A},\mathcal{G}).
\end{equation}

\end{theorem}

\begin{proof}
It is sufficient to prove  \eqref{excision-for-cs} which can be  rewritten as
$$\alpha_{{\rm ex}}([V^b_{\D}]-[e_1])=[V_{\D}, e_1, V_{(t\D_{\cyl})}]\,.$$
Indeed, if we can give an argument justifying this equality, then we can
also prove that
$$\beta_{{\rm ex}}([V^b_{\D^\otimes}]-[e_1])= [V_{\D^\otimes}, e_1, V_{(t\D^\otimes_{\cyl})}]\,;$$
on the left we have $\beta_{{\rm ex}}(\Ind_{\infty} (\D))$ whereas on the right
we have $\Theta_{{\rm rel}} ([V_{\D}, e_1, V_{(t\D_{\cyl})}])$ which is
precisely $ \Theta_{{\rm rel}}(\operatorname{CS}_\infty (\D,\D_{\pa}))$ ; thus
$$\beta_{{\rm ex}}(\Ind_{\infty} (\D))=
  \Ind (\D,\D_\infty)\,.$$
as required.

In order to show the equality
$\alpha_{{\rm ex}}([V^b_{\D}]-[e_1])=[V_{\D}, e_1, V_{(t\D_{\cyl})}]$
we can adapt the proof of the equality
\begin{equation}\label{excision-graph}
\alpha_{{\rm ex}}([e^b_{\D}]-[e_1])=[e_{\D}, e_1, e_{(t\D_{\cyl})}],
\end{equation}
 given in \cite{mp}, keeping in mind the remark given right after \eqref{s+w}.
 It is here that the specific structures of $e^b_{\D}$ and $V^b_{\D}$ are used.
 Since the details are  elementary but somewhat lengthy we omit them.
 \end{proof}

\section{Cyclic cocycles and higher indices}\label{section:higher-indeces}

Given a group cohomology class $\xi$ of $H^k (\Gamma; \CC)$,
we choose a representative cocycle   $c$  in $C^k(\Gamma; \CC)$.
Thus, see Lemma \eqref{lemma:group-cohom} and the discussion
preceding it,  $c$ is normalized,  namely:
$c(g_1, g_2, \dots , g_k)= 0$ if any $g_i=1$ or  $g_1 g_2 \cdots g_k=1$.
 Consider the algebra $\J$ with $r=1$; an element $S\in \J$
 is, in particular, a continuous section of the bundle ${\rm END} (E)\otimes \B_\infty$  on $M\times M$.
Equivalently, from the inclusion $\Psi^{-\infty,\epsilon} (M,E\otimes\B_\infty)\subset
\Psi^{-\infty,\epsilon} (M,E\otimes C^*_r \Gamma)$ and the fact  that $C^*_r \Gamma$
is contained in $\ell^2 (\Gamma)$, we can see that
$S$ is a function on $\Gamma$ with values in
$ \Psi^{-\infty,\epsilon} (M,E) $, denoted
$\Gamma\ni g \to S(g)\in  \Psi^{-\infty,\epsilon} (M,E)$.
We shall have to be precise about the continuity properties of this function, but for the time being we work on the dense subalgebra $\J_f$ of $\J$ given by the elements
of compact support in $\Gamma$; put it differently we work with the algebraic tensor product
$$\J_f:= \Psi^{-\infty,\epsilon} (M,E)\otimes \CC\Gamma \subset \J\,.$$
Before passing to the next definition, recall that elements in
$\Psi^{-\infty,\epsilon} (M,E)$ are trace class on $L^2$. Hence it makes sense to give the following

\begin{definition}
For  $S_i\in \J_f$ we set
$$
\tau_{c} (S_0 + \omega \cdot 1 , S_1, \dots S_k)
=\sum_{g_0 g_1 \cdots g_k=1}\, \Tr (S_0 (g_0) S_1 (g_1) \cdots  S_k (g_k))
c (g_1,g_2,\dots,g_k).
$$
\end{definition}

\noindent
We know, see \cite{Co}, that $\tau_c$ defines a
 cyclic cocycle for $\J_f$.

\begin{assumption}\label{assumption}(Extendability)
$\tau_c$ extends from $\J_f$ to $\J$.
 \end{assumption}

 \begin{proposition}\label{prop:classic}
 If $\Gamma$ is Gromov hyperbolic then we can choose a representative $c$ of $\xi$
 so that $\tau_c$ extends.  \end{proposition}

 \begin{proof}
 This will be proved later, see Proposition \ref{prop:extension} and its proof.
 \end{proof}

 Recall the pairing between $K$-groups and cyclic cohomology
groups and more
particularly between the $K_0$-group and the cyclic cohomology group of even degree. See the definition in \eqref{paring-abs}.

 \medskip
 \begin{definition}\label{def:higher-indeces}
 If $\tau_c$ satisfies the extendability assumption then we define the higher index
 associated to $c$
 as
 \begin{equation}\label{eq:higher-indeces}
 \Ind_{(c,\Gamma)} (\D):= \langle \Ind_\infty (\D),\tau_c \rangle
 \end{equation}
  \end{definition}


\smallskip
\noindent
We can now state the following:

\bigskip
\noindent
{\it The main goal of this paper is to prove a  Atiyah-Patodi-Singer formula
for $ \Ind_{(c,\Gamma)} (\D)$.}

\bigskip
\noindent
To this end we recall one of the main steps in the proof of the higher index theorem
of Connes-Moscovici. Let $N$ be a closed compact manifold and $\widetilde{N}\xrightarrow{\pi} N$ a Galois $\Gamma$-covering.
Consider the idempotent  $V_{\D^\otimes}$ and the index class
$\Ind_{\infty} (\D)=[V_{\D^\otimes}]-[e_1]\in K_0 (\Psi^{-\infty}(N,E\otimes (\B^\infty\otimes\CC^r))$. For $\tau_c$ extendable
and of degree $k$, $k=2p$, we have:  \begin{equation*}\Ind_{(c,\Gamma)} (\D)= \const_k\,\tau_c (V_{u\D^\otimes},\dots,V_{u\D^\otimes}), \end{equation*}
where
\begin{equation}\label{const_k}
\const_k=(-1)^p \frac{(2p)!}{p!},\quad k=2p
\end{equation}
and  $u>0$.
The following Proposition is  crucial and employs Getzler-rescaling
 in an essential way.\\Recall the data needed in order to construct the map
 \eqref{isometric-u}:
\begin{itemize}
\item A good open cover $\mathcal{U}=\{U_1,\dots,U_r\}$.
\item Continuous sections $s_i \colon U_i \to \widetilde{N}$ of the projection $\widetilde{N} \to N$.
\item A partition of unity $\chi_i$, $\supp \chi_i \subset U_i$, $\sqrt{\chi_i}$ smooth.
\end{itemize}
Given a $\Gamma$-cocycle $c$ of degree $k$ we can use this data in order to  construct a closed differential form $\omega_c$ as follows.
For every $i$, $j$ let $g_{ij}\in \Gamma$ be the unique element such that $g_{ij} s_j(x) =s_i(x)$ for every $x \in U_i\cap U_j$.
Then set
\begin{equation*}
\omega_c =  \sum \limits_{i_0,i_1, \ldots, i_k} c(g_{i_1i_2}, \ldots, g_{i_ki_0})\chi_{i_0} d\chi_{i_1} \ldots d \chi_{i_k} = \sum \limits_{i_0,i_1, \ldots, i_k} c(g_{i_0i_1}, \ldots, g_{i_{k-1}i_k})  \chi_{i_0} d\chi_{i_1} \ldots d \chi_{i_k}
\end{equation*}
The form $\omega_c$ defined by the above equation is closed and
$[\omega_c]=\nu^*  [c]$  where $\nu \colon N\to B\Gamma$
is the classifying map. Here we use the isomorphism $H^\bullet (B\Gamma,\CC)\simeq
H^\bullet(\Gamma,\CC)$.
We can give another description of the form $\omega_c$ as in \cite{LottI}. The sections $s_i$ induce diffeomorphisms $s_i \colon U_i \to \widetilde{U}_i =s_i(U_i) \subset \widetilde{N}$. One then constructs the functions $\tilde{\chi_i} \in C^\infty_0(\widetilde{U}_i)$ by
$\tilde{\chi_i} = (s_i^{-1})^* \chi_i$. Set $h = \sum_i \tilde{\chi_i} \in C^\infty_0(\widetilde{N})$. Then the function $h$ has the property that
\[
\sum_{g \in \Gamma} g\cdot h =1,
\]
where $g \cdot f(x) = f(g^{-1}x)$. Let
$\widetilde{\omega}_c\in \Omega^*(\widetilde{N})$ be the differential form given by
\[\widetilde{\omega}_c=\sum_{g_i\in\Gamma} d({g_1}\cdot h) \cdots d({g_k}\cdot h)  c(g_1, g_1^{-1}g_2, \dots,g_{k-1}^{-1}g_k)\,.\]
This form is $\Gamma$-equivariant and moreover  $\widetilde{\omega}_c=\pi^* (\omega_c)$.

\begin{proposition}\label{prop:short-time}
For any $u>0$ we have
\begin{equation}\label{short-time}
  \const_k \cdot \tau_{c} (V_{u\D^\otimes},\dots,V_{u\D^\otimes})=\int_{N} {\rm AS}\wedge \omega_{c}\,.
\end{equation}
with ${\rm AS}$ equal to the Atiyah-Singer integrand.
\end{proposition}

This theorem has been proved by Connes-Moscovici in \cite{CM}. In that paper they used the Getzler's calculus  to compute short-time asymptotics
of Wassermann projection. In \cite{MWU} Moscovici and Wu show that the same method applies to a wider class of idempotents,
in particular to the symmetrized idempotent
\[\widetilde{V}_{u\D^\otimes}:=V_{u\D^\otimes} \oplus V_{u\D^\otimes}^*.
\]
 This is the result that we will need for the calculation
of the short-time limit. We note that a different method, also based on Getzler's rescaling, but using instead the superconnection techniques, has been used
by J. Lott in \cite{LottI}.

We recall now some of the steps in Connes-Moscovici's proof this theorem, using slight modification from \cite{MWU}, referring the reader to
\cite{CM}, \cite{MWU}, \cite{LottI} for  details.
We start by noticing that as $\tau_c$ extends to pair with the $K$-theory of $C^*_r\Gamma$, we have
\[
\tau_{c} (V_{u\D^\otimes},\dots,V_{u\D^\otimes}) = \frac{1}{2} \tau_{c} (\widetilde{V}_{u\D^\otimes},\dots,\widetilde{V}_{u\D^\otimes})
\]

Consider the cochain $\tilde{\tau}_{c}$ on the smoothing operators $\Psi^{-\infty}(N)$ given by
\begin{equation*}
\tilde{\tau}_{c}(A_0, A_1, \ldots, A_k)= \int_{N^{k+1}}\tr A_0(x_0, x_1) \ldots A_k(x_k, x_0) \phi_c(x_0,\ldots, x_k)dx_0\ldots dx_k
\end{equation*}
where
\begin{equation*}
\phi_c(x_0,\ldots, x_k) =  \sum \limits_{i_0,i_1, \ldots, i_k} c(g_{i_1i_2}, \ldots, g_{i_ki_0})\chi_{i_0}(x_0) \chi_{i_1}(x_1) \ldots  \chi_{i_k}(x_k).
\end{equation*}
For $k >0$ $\tilde{\tau}_{c}$ is extended to the unitalization of $\Psi^{-\infty}(N)$ by $\tilde{\tau}_{c}(A_0, A_1, \ldots, A_k)=0$ if one of $A_i=1$.
To prove the proposition  one first establishes equality
\begin{equation*}
\tau_{c} (\widetilde{V}_{u\D^\otimes},\dots,\widetilde{V}_{u\D^\otimes})=\tilde{\tau}_{c}(\widetilde{V}_{uD}, \widetilde{V}_{uD}, \ldots, \widetilde{V}_{uD}) +O(u^\infty) \text{ as } u \to 0
\end{equation*}
where $D$ is the Dirac operator on $N$.
(Notice that an inspection of the arguments in \cite{CM} shows that the trace identity is not used in this
proof; this will be important when we shall want to extend this result to $b$-manifolds.)
In the next step one uses Getzler's calculus to show that
\[\lim_{u\to 0} \const_k \cdot \tilde{\tau}_{c}(\widetilde{V}_{uD}, \widetilde{V}_{uD}, \ldots, \widetilde{V}_{uD})  =2\int_{N} {\rm AS}\wedge \omega_{c}.\]
In fact, Connes and Moscovici (for the case of Wassermann projection) and Moscovici and Wu obtain a local result, computing the limit of the corresponding trace density. Later in the paper we shall deal with the case of manifolds with cylindrical ends.

\medskip

We end this Section by discussing the compatibility of our definition with
the one appearing in the work of the third author and Leichtnam.
For the latter we consider the Mishchenko-Fomenko index class
$\Ind_{{\rm MF},\infty} (\D) \in K_* (\B^\infty)$. Recall that this is obtained
through a $\B^\infty$-MF decomposition theorem; thus there exist
 finitely generated projective $\B^\infty$-submodules
 $\L_{\infty}\subset H^\infty_b (M,
 E^+\otimes\V^\infty)$ and $\N_{\infty}\subset H^\infty_b (M,
 E^-\otimes\V^\infty)$, with $H^\infty_b:=\cap_{k\in\NN} H^k_b$,
 and decompositions
 $$\L_{\infty}\oplus \L_{\infty}^\perp= H^\infty_b (M,
 E^+\otimes\V^\infty)\quad\quad \N_{\infty}\oplus \D^+ (\L_{\infty}^\perp)= H^\infty_b (M,
 E^-\otimes\V^\infty)$$
 so that $\D^+$ is block diagonal and invertible when restricted to  $\L_{\infty}^\perp$.
 We refer the reader to  \cite[Theorem 12.7]{LPMEMOIRS} and \cite[Appendix]{LPBSMF} for the precise statement.
The Mishchenko-Fomenko $\B^\infty$-index is, by definition,
$$\Ind_{{\rm MF},\infty}(\D):=[\L_\infty] - [\N_\infty]\in K_* (\B^\infty).$$
We can consider the Karoubi-Chern character of this class, with values in the noncommutative de Rham homology of $\B^\infty$:
$$\Ch_{K} (\Ind_{{\rm MF},\infty}(\D))\in\overline{H}_\bullet (\B^\infty)\,.$$
Fix $c\in Z^{2p}(\Gamma;\CC)$,
a normalized group cocycle with associated reduced cyclic cocycle $t_c\in \overline{C}_{\lambda}^{2p}(\CC\Gamma) \subset CC^{2p}(\CC\Gamma)$.
Assume  now that $\Gamma$ satisfies the (RD) condition and
that $c$ is of polynomial growth; then $t_c$ extends from $\CC\Gamma$
to $\B^\infty$ and since $\overline{H}_\bullet (\B^\infty)$ can be paired with
(reduced) cyclic cohomology, see Subsection \ref{subsect:karoubi},
we obtain a number
$\langle \Ch_K (\Ind_{{\rm MF},\infty}(\D)),t_c \rangle_K$.

\begin{proposition}\label{prop:compatibility}
Under the above assumptions on $\Gamma$ and $c$, and with the notation introduced so far, the following equality holds:
\begin{equation}\label{compatibility}
\Ind_{(c,\Gamma)}(\D)=\langle \Ch_K (\Ind_{{\rm MF},\infty}(\D)),t_c \rangle_K\,.
\end{equation}
\end{proposition}

\begin{proof}
Let $\Pi_+$ be the orthogonal projection onto $\L_\infty$ and let $\Pi_-$
be the projection onto $\N_\infty$ along $\D( \L_\infty^\perp)$. It is proved
in \cite[Theorem 12.7]{LPMEMOIRS} that these elements are residual.
Thus $$P:=
\left(\begin{array}{cc} \Pi_+ & 0\\ 0 &   \Id-\Pi_-
 \end{array} \right)\in \mathfrak{J}^+\,,$$
 with $\mathfrak{J}^+$ denoting the unitalization of $\mathfrak{J}$ and
 $[P]-[e_1]\in K_0 (\mathfrak{J})$. By choosing as
a parametrix of $\D^+$ the Green operator defined by the Mishchenko-Fomenko
decomposition, i.e. the operator equal to 0 on $\N_\infty$ and equal to  $(\D^+|_{\L_\infty^\perp})^{-1}$ on $\D^+(\L_\infty^\perp)$, we easily see that
$$\operatorname{CS}_\infty (\D)= [P]-[e_1] \in K_0 (\mathfrak{J})\,.$$
Thus
$$\Ind_\infty (\D)= \left[ \left(\begin{array}{cc} \theta(\Pi_+) & 0\\ 0 &   \Id-\theta(\Pi_-)
 \end{array} \right)\right] - \left[ \left(\begin{array}{cc} 0& 0\\ 0 &   \Id
 \end{array} \right)\right]$$
 which implies
 $$\Ind_{(c,\Gamma)} (\D)= \langle \left[ \left(\begin{array}{cc} \theta(\Pi_+) & 0\\ 0 &   \Id-\theta(\Pi_-)
 \end{array} \right)\right] - \left[ \left(\begin{array}{cc} 0& 0\\ 0 &   \Id
 \end{array} \right)\right],\tau_c \rangle\,.$$
 Recall the isometric embedding $U$, see \eqref{isometric-u}, that we rewrite
 in the $b$-context as
$H^\infty_b (M,E\otimes \V^\infty)\xrightarrow{U}
H^\infty_b (M, E\otimes( \B^\infty\otimes \CC^k))
$;
 this identifies $\L_\infty$ and $\N_\infty$ with two finitely generated projective
$\Bi$-modules $\L_\infty^\otimes $ and $\N_\infty^\otimes $ in  $H^\infty (M, E^\pm\otimes( \B^\infty\otimes \CC^k))$ and $\theta (\Pi_\pm)$ are projections
onto  $\L_\infty^\otimes $ and $\N_\infty^\otimes $.
There are natural connections  on these finitely generated projective modules,
obtained by compressing with $\theta (\Pi_\pm)$ the trivial connection $d_\Gamma$
induced by the
differential  in the $\Gamma$-direction,
$d: \B^\infty\to \Omega_1 (\B^\infty)$.
Thus we can compute the right hand side of \eqref{compatibility} by using
$\L_\infty^\otimes $ and $\N_\infty^\otimes $ endowed with the connections
$\theta (\Pi_\pm) \,d_\Gamma \,\theta (\Pi_\pm)$.
Recall now that the definition of the pairing between $K_0 (\J)$ and $HC^{2\star}
(\J)$ is through the Connes-Chern character from $K$-theory to cyclic homology,
 see \eqref{pairing-cyclic-coh}
 \begin{equation}\label{explicit}
 \langle \left[ \left(\begin{array}{cc} \theta(\Pi_+) & 0\\ 0 &   \Id-\theta(\Pi_-)
 \end{array} \right)\right] - \left[ e_1\right],\tau_c \rangle :=
 \langle \Ch \left[ \left(\begin{array}{cc} \theta(\Pi_+) & 0\\ 0 &   \Id-\theta(\Pi_-)
 \end{array} \right)\right] - \left[ e_1\right],\tau_c \rangle_{HC}
 \end{equation}
 where  $e_1=\left(\begin{array}{cc} 0& 0\\ 0 &   \Id
 \end{array} \right)$ and where we recall that
 given an idempotent $p$ in  $\J$ one defines
 $$\Ch (p)=p+\sum_{k\geq 1} (-1)^k \frac{(2k)!}{k !} (p-\ha)\otimes p^{\otimes 2k}$$
 and similarly for an idempotent $p$ in $M_{r\times r} (\J)$.
 What appears above, in \eqref{explicit}, is the left hand side of \eqref{compatibility};
 unwinding this expression one can show easily that the number we get
 is precisely equal to
 $\langle \Ch (\L_\infty^\otimes),t_c \rangle_K -
\langle \Ch (\N_\infty^\otimes),t_c \rangle_K$
with the first Chern character computed with the connection $\theta (\Pi_+) \,d_\Gamma \,\theta (\Pi_+))$ and the second one with $\theta (\Pi_-) \,d_\Gamma \,\theta (\Pi_-)$. Since, as just explained,
this is in turn equal to the right hand side of \eqref{compatibility}, we conclude that
the proof of the proposition is  complete.

\end{proof}



\section{The relative cyclic cocycle $(\tau^r_c,\sigma_c)$ associated
to a group cocyle}\label{sect:relative-cocycles}
  Consider now the algebra $\A$ with $r=1$; an element $A\in \A$
  is a function on $\Gamma$ with values in
$\Psi^{-\infty,\epsilon}_b (M,E) + \Psi^{-\infty,\epsilon} (M,E) $, denoted
$\Gamma\ni g \to A(g)\in \Psi^{-\infty,\epsilon}_b (M,E) + \Psi^{-\infty,\epsilon} (M,E)$.
We first work on the dense subalgebra $\A_f$ of $\A$ given by the elements of compact
support in $\Gamma$, i.e.
$$\A_f:=(\Psi^{-\infty,\epsilon}_b (M,E) + \Psi^{-\infty,\epsilon} (M,E))\otimes \CC\Gamma.$$

\begin{definition}\label{def:regularized-tau}
For  $A_i\in \A_f$ we set
\begin{equation}\label{regularized-tau}
\tau^r_{c} (A_0 + \omega \cdot 1 , A_1, \dots A_k)
=\sum_{g_0 g_1 \cdots g_k=1}\, {}^b\Tr (A_0 (g_0) A_1 (g_1) \cdots  A_k (g_k))
c (g_1,g_2,\dots,g_k).
\end{equation}
\end{definition}

Recall 
the definition of a double complex $(C^*(\A), B+b)$
 for an arbitrary algebra $\A$ over $\CC$;
as already explained, the cochain complex $(CC^n (\A), B+b),$  consists of
multilinear mappings $\tau: \A^+\otimes \A^{\otimes n} \to \CC$ with the Hochschild coboundary map
$
b: C^n(\A) \to C^{n+1}(\A)
$
and
$
B:C^{n+1}(\A) \to C^n(\A)
$

\begin{lemma}
In the double complex $(CC^*(\A_f), b+B)$ one has
$B \tau^r_{c} =0$.
\end{lemma}

\begin{proof}
This is obvious.
\end{proof}

Consider now
$\G_f$
which is nothing but the algebraic tensor product $G\otimes \CC\Gamma$,with
$G=\Psi^{-\infty,\epsilon}_{b,\RR^+} (\overline{N^+ (\pa M)},E) +
\Psi^{-\infty,\epsilon}_{\RR^+}  (\overline{N^+ (\pa M)},E)$.
Recall that there exists a (surjective) homomorphism
$I : \mathcal{A}_f \to \mathcal{G}_f $, the indicial operator.

\begin{lemma}\label{lemma:sigmac}
Let $\sigma_c$  be the cochain on $\mathcal{G}_f$ defined
by
\begin{align*}
&\sigma_c (B_0 +\omega\cdot 1,B_1, \dots, B_{k+1})\\
&:=(-1)^{k+1} \sum_{g_0\cdots g_{k+1}=1}
 \frac{i}{2\pi}\int d\lambda  \Tr \left( (\widehat{B}_0 (\lambda)(g_0)
\cdots \widehat{B}_k (\lambda)(g_k) \frac{d \widehat{B}_{k+1}(\lambda)(g_{k+1})}{d\lambda} \right)
c(g_1,g_2, \dots,g_{k})
\end{align*}
Then
$b \tau^r_{c} =I^* \sigma_c$. \\
We call $\sigma_c$ the {\bf eta cocycle} associated to $c$.
\end{lemma}

\begin{proof}
We observe first of all that
\begin{align*}
& \tau^r_c (A_0 A_1,A_2,\dots,A_{k+1})=\sum_{\gamma g_2\cdots g_{k+1}=1}
{}^b \Tr ((A_0 A_1) (\gamma) A_2 (g_2) \cdots A_{k+1} (g_{k+1})) c(g_2,g_3,
\dots,g_{k+1})\\
&=\sum_{\gamma g_2\cdots g_{k+1}=1;\gamma=g_0 g_1}
{}^b \Tr (A_0 (g_0) A_1 (g_1) A_2 (g_2) \cdots A_{k+1} (g_{k+1})) c(g_2,g_3,
\dots,g_{k+1})\\
&= \sum_{g_0 g_1 g_2\cdots g_{k+1}=1}
{}^b \Tr (A_0 (g_0) A_1 (g_1) A_2 (g_2) \cdots A_{k+1} (g_{k+1})) c(g_2,g_3,
\dots,g_{k+1})
\end{align*}
We also observe that
\begin{equation*}
 \tau^r_c (A_0,\dots,A_i A_{i+1}, \dots, A_{k+1})=\sum_{g_0\cdots g_{k+1}=1}
{}^b \Tr (A_0 (g_0)  \cdots A_{k+1} (g_{k+1})) c(g_1,
\dots,g_i g_{i+1}, \dots, g_{k+1})\end{equation*}
and that
\begin{align*}
& \tau^r_c (A_{k+1} A_0, A_1,\dots,A_k)=\sum_{ g_0\cdots g_{k+1}=1}
{}^b \Tr (A_{k+1} (g_{k+1}) A_0 (g_0) A_1 (g_1) \cdots A_{k} (g_{k})) c(g_1,g_2,
\dots,g_{k})\\
&=\sum_{g_0\cdots g_{k+1}=1}
{}^b \Tr (A_0 (g_0) A_1 (g_1)  \cdots A_{k+1} (g_{k+1})) c(g_1,g_2,
\dots,g_{k}) + \\
& \sum_{g_0\cdots g_{k+1}=1} {}^b \Tr [ A_{k+1} (g_{k+1}), A_0 (g_0) A_1 (g_1)  \cdots A_k (g_k)]
 c(g_1,g_2,
\dots,g_{k})
\end{align*}
Adding up and using the fact that $c$ is a cocycle we see that
$$b \tau^r_c (A_0, \dots, A_{k+1}) = (-1)^{k+1}
\sum_{g_0\cdots g_{k+1}=1} {}^b \Tr [ A_{k+1} (g_{k+1}), A_0 (g_0) A_1 (g_1)  \cdots A_k (g_k)] c(g_1,g_2, \dots,g_{k}) $$
Thus, using the $b$-trace identity of Melrose we find:
$$b \tau^r_c (A_0, \dots, A_{k+1}) =
(-1)^{k+1} \sum_{g_0\cdots g_{k+1}=1}
\frac{i}{2\pi}\int d\lambda \Tr \left(  I(A_0,\lambda)(g_0)
\cdots I(A_k,\lambda)(g_k) \frac{d I(A_{k+1},\lambda)(g_{k+1})}{d\lambda} \right)
c(g_1,g_2, \dots,g_{k})
$$
We end the proof by computing $b \tau^r_c (1,A_1, \dots, A_{k+1})$.

We have:
\begin{align*}
& b \tau^r_c (1,A_1, \dots, A_{k+1})= \tau^r_c (A_1, \dots,A_{k+1}) + (-1)^{k+1}
\tau^r_c (A_{k+1},A_1, \dots,A_k)\\
&= \sum_{g_1\cdots g_{k+1}=1}  {}^b \Tr (A_1 (g_1) \cdots A_{k+1} (g_{k+1}))
c(g_2, \dots,g_{k+1}) +\\& (-1)^{k+1} \sum_{g_1\cdots g_{k+1}=1}
{}^b \Tr (A_{k+1} (g_{k+1}) A_1 (g_1) \cdots A_{k} (g_{k}))
c(g_1, \dots,g_{k})\\
&= (-1)^{k+1} \sum_{g_1\cdots g_{k+1}=1}
{}^b \Tr [A_{k+1} (g_{k+1}), A_1 (g_1) \cdots A_{k} (g_{k})]
c(g_1, \dots,g_{k})\\
&= (-1)^{k+1} \sum_{g_1\cdots g_{k+1}=1} \frac{i}{2\pi}\int d\lambda \Tr \left( I(A_1,\lambda)(g_1)
\cdots I(A_k,\lambda)(g_k) \frac{d I(A_{k+1},\lambda)(g_{k+1})}{d\lambda} \right)
c(g_1,g_2, \dots,g_{k})
\end{align*}
where we have used Lemma \ref{lemma:group-cohom}, part \ref{c} in the penultimate step. The Lemma is
proved.

\end{proof}

\begin{lemma}
For the cochain $\sigma_c$ we have
\begin{equation*} b\sigma_c=0,\ B\sigma_c=0\end{equation*}
\end{lemma}

\begin{proof}
This is an immediate consequence of the definitions and of the
previous Lemma, given that
$I^*$ is injective.
\end{proof}

\noindent
Summarizing, we have proved  the following:
 $(\tau^r_c,\sigma_c)\in C^k(\A_f)\oplus C^{k+1} (\G_f)$ and
 $$\left( \begin{array}{cc} b+B & -I^* \\ 0&-(b+B)
\end{array} \right) \left( \begin{array}{c} \tau^r_c \\ \sigma_c
\end{array} \right)=\left( \begin{array}{c} 0 \\ 0
\end{array} \right)
$$
We restate all this in the following important

\begin{theorem}
Let $c$ be a normalized group cocycle for $\Gamma$.
The cochain $\tau^r_c$ defined in \eqref{def:regularized-tau}
and the cochain $\sigma_c$ defined in Lemma \ref{lemma:sigmac}
define  together a
  relative cyclic cocycle $(\tau^r_{c}, \sigma_c)$ for $\mathcal{A}_f \xrightarrow{I} \mathcal{G}_f$.
\end{theorem}

\section{Continuity properties of the cocycles $(\tau^r_c,\sigma_c)$ and $\tau_c$}
\label{sect:continuity}
Let us recall the definition of  the Connes-Moscovici algebra $\B^\infty\subset C^*_r \Gamma$. See \cite{CM} and \cite{WuI} for more details.
Fix a word metric $| \cdot |$ on $\Gamma$.
Define an unbounded operator $D$ on $\ell^2(\Gamma)$ by setting
$D(e_\gamma)= |\gamma | e_\gamma $ where $(e_\gamma)_{\gamma \in
\Gamma}$ denotes the standard orthonormal basis of
$\ell^2(\Gamma)$. Then consider the unbounded derivation  $
\delta( T)= [D, T]$ on ${B} (\ell^2(\Gamma))$ and set
$$\mathcal{B}^\infty=\{ T \in C^*_r(\Gamma) \: | \; \forall k \in \NN,
\; \delta^k(T) \in {B}(\ell^2(\Gamma)) \}.$$ It is not difficult to prove that
$\CC\Gamma\subset \mathcal{B}^\infty$ and that $\mathcal{B}^\infty$ is dense
in $C^*_r \Gamma$.
We endow $\mathcal{B}^\infty$ with the topology defined by the restriction of
the $C^*_r \Gamma$-norm and the sequence of seminorms
\begin{equation}\label{seminorms-b-infinity}
\| T \|_j := \| \delta^j (T) \|
\end{equation}
where on the right hand side we have the operator norm in ${B}(\ell^2(\Gamma))$.
$\mathcal{B}^\infty$ is  a Fr\'echet (locally $m$-convex) algebra and it is
 closed under
holomorphic functional calculus in $C^*_r \Gamma$.

For the continuity properties of the relative cocycle $(\tau_c^r,\sigma_c)$
and of the cocycle $\tau_c$
 we shall not work directly with the seminorms defining $\mathcal{B}^\infty$  but will employ  instead  norms $\nu_k (\cdot)$ on $C^*_r \Gamma$ which, as proved in \cite[Lemma (6.4) (i)]{CM},  is continuous
 on $\mathcal{B}^\infty$. Let us recall the definition:

 \noindent
if $a\in C^*_r \Gamma$ and $k\in \NN$ we define
 \begin{equation}\label{norms-nu-k}
 \nu_k (a) = \left( \sum_{g\in\Gamma} (1+|g|)^{2k} | a(g) |^2 \right)^{1/2}
 \end{equation}
 Recall also that a  finitely generated discrete group  $\Gamma$ satisfies the rapid decay condition
 (RD) if there exists $k\in \NN$ and $C>0$ such that
 $$\| a \| ^2_{C^*_r (\Gamma)} \leq C \sum_{g\in \Gamma}  (1+|g|)^{2k} | a(g) |^2\,,\quad
 \forall a\in\CC\Gamma.$$
 It is a non-trivial result that Gromov hyperbolic groups satisfy the (RD) condition;
 moreover, for each $\xi\in H^\bullet (\Gamma;\CC)$ there exists a polynomially bounded 
 cocycle $c\in Z^\bullet (\Gamma;\CC)$ such that $\xi=[c]$.

 The main goal of this section is to establish the following
 proposition:

 \begin{proposition}\label{prop:extension}
 If $\Gamma$ is a finitely generated discrete group satisfying the rapid decay condition
 (RD) and $c\in Z^k (\Gamma;\CC)$ has polynomial growth with respect to a word metric $|\cdot |$ then
\begin{equation}\label{extend-tau-r}
\tau_c^r\;\; \text{extends continuously from}\;\; \A_f\;\; \text{to}\;\; \A\,;
\end{equation}
\begin{equation}\label{extend-sigma}
\sigma_c\;\; \text{extends continuously from}\;\; \G_f\;\; \text{to}\;\; \G\,.
\end{equation}
\begin{equation}\label{extend-tau}
\tau_c\;\; \text{extends continuously from}\;\; \J_f\;\; \text{to}\;\; \J\,.
\end{equation}
Moreover, the extended pair $(\tau_c^r,\sigma_c)$ is a relative cyclic cocycle
for $\A\xrightarrow{I} \G$.
 \end{proposition}

\begin{proof}
The last statement follows by continuity from the corresponding statement
for  $\A_f\xrightarrow{I} \G_f$. We thus turn to  \eqref{extend-tau-r}, \eqref{extend-sigma}, \eqref{extend-tau}.
The main difficulty in establishing \eqref{extend-tau-r} comes from the use of the
 $b$-Trace in the definition of $\tau^r_c$. Crucial in our argument will be the following
 Proposition, due to Lesch, Moscovici and Pflaum, see \cite[Proposition 2.6]{LMP2}.

 Before stating it, we introduce some notation. Let $\phi\in C^\infty (M)$ be a function equal to $t$ on the cylindrical end $(-\infty, 0]\times \pa M_0\subset M$. Let $\mathcal{V}$ be a vector field   equal to
 $\pa/\pa t$ on the cylindrical end.
  In particular   $ \mathcal{V} (\phi)=1$ on the cylindrical end.  Let $\chi:= 1-\mathcal{V} (\phi) \in C^\infty_c (M_0\setminus \pa M_0)$.

 \begin{proposition} (Lesch-Moscovici-Pflaum)
 If $P\in A:= \Psi^{-\infty,\epsilon}_b (M,E) + \Psi^{-\infty,\epsilon} (M,E) $ then

 \begin{equation}\label{LMP}
  {}^b\Tr (P)= -\Tr (\phi [\mathcal{V},P]) + \Tr (\chi P)\,.
  \end{equation}
 \end{proposition}
 Consequently, the $b$-Trace of $P$ is the difference of the traces of two
 trace-class operators
 naturally associated to $P$.  On the basis of this Proposition and of the next Lemma (in particular
 its proof), we give the following

 \begin{definition}\label{def:triple}
 If $P\in A$, with $A:= \Psi^{-\infty,\epsilon}_b (M,E) + \Psi^{-\infty,\epsilon} (M,E) $
 then
 \begin{equation}\label{triple}
 ||| P |||^2 := \| \chi P \|^2_1 + \| \phi [\mathcal{V},P] \|^2_1 + \| [\mathcal{V},P] \|^2_1
 + \| [\phi ,P] \|^2 + \| P \|^2
 \end{equation}
 with the last two norms denoting the $L^2$-operator norm.
 \end{definition}

 \begin{lemma}\label{lemma:lemma1}
 If $P_j\in \Psi^{-\infty,\epsilon}_b (M,E) + \Psi^{-\infty,\epsilon} (M,E)$, $j\in \{0,1,\dots,k\}$, then there exists $C>0$ such that
 \begin{equation}\label{inequality-lemma1}
 |   {}^b\Tr  (P_0 P_1 \cdots P_k)| \leq  C ||| P_0 ||| \cdots ||| P_k |||
 \end{equation}
 \end{lemma}

 \begin{proof}
 Using formula \eqref{LMP} we see that
 \begin{align*}
& |   {}^b\Tr  (P_0 P_1 \cdots P_k)|  \\
&\leq   | \Tr (\phi  [\mathcal{V},P_0 P_1 \cdots P_k])| + |\Tr (\chi P_0 P_1 \cdots P_k)| \\
&\leq  \sum_i | \Tr (\phi P_0 \cdots [ \mathcal{V},P_i] \cdots P_k)|  + |\Tr (\chi P_0 P_1 \cdots P_k)|\\
& \leq \sum_{j<i} | \Tr (\phi P_0 \cdots [\phi ,P_j]\cdots [ \mathcal{V},P_i] \cdots P_k)|
+ \sum_i | \Tr (P_0 \cdots \phi [\mathcal{V},P_i] \cdots P_k) | +
|\Tr (\chi P_0 P_1 \cdots P_k)|\\
& \leq  \sum_{j<i} \|P_0\| \cdots \widehat{j}\cdots \widehat{i}
\cdots \|P_k\| \| [\phi,P_j]\| \| [\mathcal{V},P_i]\|_1 + \sum_i \| P_0\| \cdots \|P_{i-1}\| \| \phi [\mathcal{V},P_i]\|_1 \cdots \|P_k\| +\|\chi P_0\|_1 \|P_1\|\cdots \|P_k\|\\
&\leq C ||| P_0 ||| \cdots ||| P_k |||\,. \end{align*}

 \end{proof}
We now introduce norms on $\mathcal{A}_f:= \Psi^{-\infty,\epsilon}_b (M,E\otimes \CC\Gamma) + \Psi^{-\infty,\epsilon} (M,E\otimes \CC\Gamma)$. If $\P\in \mathcal{A}_f$ then, as already remarked,
$$\P=\sum_{g\in \Gamma} P(g) g$$
with $P(g)\in A$ and where  the sum is finite.
We set
\begin{equation}\label{triple-k}
||| \P |||^2_k := \sum_{g\in \Gamma} ||| P(g) |||^2 (1+ |g|)^{2k}
\end{equation}

\begin{lemma}\label{lemma:lemma2}
Let $k\in\NN$. If $\P\in\A$ then $||| \P |||_k < \infty$.
Consequently, $\A$ is contained is the closure  of $\A_f$ with respect
to the norm $|||\cdot|||_k$.
\end{lemma}
We now recall the following fundamental result, due to Connes-Moscovici and Jolissant.
\begin{lemma}\label{lemma:lemma3}
 Let $\Gamma$ be a discrete finitely generated group satisfying the rapid decay
 condition (RD). Let $c\in Z^k (\Gamma;\CC)$ be polynomially bounded.
 Let $f_j\in\CC\Gamma$, $j\in\{0,1,\dots,k\}$. Then there exists $m\in \NN$
 and $C>0$ such that
 \begin{equation}\label{lemma3}
 \sum_{g_0\cdots g_k=1} |f_0 (g_0) f_1 (g_1) \cdots f_k (g_k)c(g_1,\dots,g_k) |
 \leq C \nu_m (f_0) \cdots \nu_m (f_k)\,.
 \end{equation}
 \end{lemma}

 Granted Lemma \ref{lemma:lemma2} we can now conclude the proof of
 \eqref{extend-tau-r}. Indeed, let $\P_0,\dots,\P_k$ be elements in $\A_f$
 and consider $f_j\in\CC\Gamma$ defined by $f_j (g) := ||| P_j (g) |||$; then
 \begin{align*}
  |\tau_c^r (\P_0,\dots,\P_k)| &\leq \sum_{g_0 g_1 \cdots g_k=1} |   {}^b\Tr (\P_0 (g_0)
  \cdots \P_k (g_k) c(g_1,\dots,g_k)|\\
  &\leq C\sum_{g_0 g_1 \cdots g_k=1} ||| \P_0 (g_0) ||| \cdots ||| \P_k (g_k)||| \;
  |c(g_1,\dots,g_k)| \\
  &\leq C^\prime \nu_m (f_0) \cdots \nu_m (f_k) = C^\prime ||| \P_0 |||_m \cdots
  ||| \P_k |||_m
  \end{align*}
where we have  used Lemma \ref{lemma:lemma1} in the second inequality
and Lemma \ref{lemma:lemma3} in the third inequality. This shows that
there exists $m\in \NN$ such that $\tau^r_c$ extends continuously
to the closure of $\A_f$ with respect to the $|||\cdot|||_m$-norm. By Lemma
\ref{lemma:lemma2} we conclude that $\tau^r_c$ extends continuously
to $\A$, which is the content of \eqref{extend-tau-r}.

A similar, easier, argument  proves \eqref{extend-tau}, the extension of $\tau_c$ from $\mathcal{J}_f$ to $\mathcal{J}$ for groups satisfying (RD) and group cocycles
that are polynomially bounded. Indeed, recall that elements in the residual calculus
are trace class; moreover the following Lemma holds:

\begin{lemma}\label{lemma:lemma2J}
Let $k\in\NN$. If $\R=\sum_{g\in\Gamma} R(g) g \in \mathcal{J}$,  then
\begin{equation}\label{lemma2J}
\sum_{g\in \Gamma} || R(g) ||_1 ^2 (1+ |g|)^{2k} < \infty
\end{equation}
Notice that this implies that
\begin{equation}\label{lemma2J-bis}
\sum_{g\in \Gamma} || R(g) || ^2 (1+ |g|)^{2k} < \infty
\end{equation}
\end{lemma}
Consequently $\mathcal{J}$ is contained in the closure of $\mathcal{J}_f$
with respect to the norms defined by the left hand side of \eqref{lemma2J}
and \eqref{lemma2J-bis}.
Adapting (in an easier situation) the arguments given above for  \eqref{extend-tau-r} we conclude that \eqref{extend-tau} holds.

\bigskip
\noindent
{\bf End of the proof of Proposition \ref{prop:extension}.}\\
We need to establish Lemma \ref{lemma:lemma2}, Lemma \ref{lemma:lemma2J}
 as well as
\eqref{extend-sigma}.\\
We shall first establish results for an element $S$ in $G$ (where we recall
that $G$ is the space  of $\RR^+$-invariant operators in the $b$-calculus
with $\epsilon$-bounds in the
compactified positive normal bundle to the boundary) and then, in the Mishchenko-Fomenko context, for an element $\mathcal{S}$ in $\mathcal{G}$.
By making the substitution
$t=\log x$, we will be equivalently looking at translation invariant operators on the infinite cylinder;
the estimates appearing in the definition of calculus with bounds translate then into weighted exponential bounds, i.e. with respect to $e^{|t|\epsilon}$,
at $t=\pm \infty$. More generally, we can consider any smooth
closed compact manifold $N$, not necessarily a boundary, and the infinite cylinder
$N\times\RR$; all the arguments that will be given below apply to this general setting.
An operator $S$ in $G$ can then be seen as a Schwartz kernel $K_S (x,y,t)$
on $N \times N\times \RR$, acting as a convolution operator in the $t$-variable.
In order to simplify the notation we shall often
write $K$, and not $K_S$, for the Schwartz kernel of $S$. We denote by $\widehat{K} (\lambda):=\mathcal{F}_{t\to
\lambda} (K)$ the
Fourier transform, in $t$, of the kernel $K$; this is a smooth family of smoothing kernels
on $N\times N$ which is rapidly decreasing, with all its derivatives,  in $\lambda$ as $\lambda\to \pm\infty$. We make a small abuse of notation and  keep the same symbol
for the smoothing kernel $\widehat{K_S}(\lambda)$ and the smoothing operator it defines
on $N$.

We begin by establishing a number of elementary results about $S$, $K_S$ and $\widehat{K_S}(\lambda)$.
First notice that $\widehat{K_S}:\RR\to B (\mathcal{H})$, with $\mathcal{H}=L^2 (N)$; the family $\widehat{K_S}$ acts in a natural way on $L^2 (\RR,\mathcal{H})$ (by
multiplication in the $\RR$ variable and by its natural action on $\mathcal{H}$) and
as such
has a norm $\| \widehat{K_S}\|_{B (L^2 (\RR,\mathcal{H}))}$. Since
Fourier transform interchanges convolution and multiplication we clearly have
\begin{equation}\label{1}
\|S\|_{L^2 (\RR_t \times N)} = \| \widehat{K_S}\|_{B (L^2 (\RR_\lambda,\mathcal{H}))}.
\end{equation}
On the other hand we observe that
\begin{equation}\label{2}
\| \widehat{K_S}\|_{B (L^2 (\RR_\lambda,\mathcal{H}))}
\leq \sup_{\lambda\in\RR} \| \widehat{K_S}(\lambda)\|_{B (\mathcal{H})}.
\end{equation}
Indeed, using $K$ instead of $K_S$, we have for any $f\in L^2 (\RR,\mathcal{H})$:
\begin{align*}
\langle \widehat{K} f,  \widehat{K} f \rangle_{L^2 (\RR,\mathcal{H})} &= \int_\RR \langle \widehat{K} (\lambda) f (\lambda),  \widehat{K} (\lambda) f (\lambda)  \rangle_{\mathcal{H}} \,d\lambda\\
&\leq \int_\RR \| \widehat{K} (\lambda)\|^2_{B (\mathcal{H})}  \|f(\lambda)\|^2_{\mathcal{H}}\, d\lambda \\
&\leq \sup_{\lambda\in\RR} \| \widehat{K}(\lambda)\|^2_{B (\mathcal{H})}  \int_\RR \|f(\lambda) \|^2_{\mathcal{H}} \,d\lambda\\
&= (\sup_{\lambda\in\RR} \| \widehat{K}(\lambda)\|^2_{B (\mathcal{H})} ) \,\|f\|^2_{L^2 (\RR,\mathcal{H})}
\end{align*}
Putting \eqref{1} and \eqref{2} together and using a well known inequality
we obtain
\begin{equation}
\|S\|_{L^2 (\RR_t \times N)}\leq \sup_{\lambda\in\RR} \| \widehat{K_S}(\lambda)\|_{B (\mathcal{H})} \leq \sup_{\lambda\in\RR} \| \widehat{K_S}(\lambda)\|_{{\rm HS}}
\end{equation}
with $\| \cdot \|_{{\rm HS}}$ denoting the Hilbert-Schmidt norm.

\noindent
Now, let $H$ any Hilbert space, for example the Hilbert space of Hilbert-Schmidt
operators on $L^2 (N)$. For any  smooth (non-vanishing) rapidly decreasing
function $\kappa:\RR\to H$ we have,
\begin{align*}
\| \kappa (\lambda) \| &=\| \int_\RR \widehat{\kappa} (\xi ) e^{i \xi \lambda} d\xi  \,\| \leq
\int_\RR \|\widehat{\kappa} (\xi) \|d\xi\\
&=\int_\RR \left \langle \widehat{\kappa} (\xi) \sqrt{1+\xi^2}, \frac{\widehat{\kappa} (\xi)}{\|\widehat{\kappa} (\xi)\|} \frac{1}{\sqrt{1+\xi^2} } \right \rangle d\xi\\
&\leq C \left( \int_\RR \|\widehat{\kappa} (\xi)\|^2 (1+\xi^2)d\xi \right)^{1/2}\\
&= C   \left( \|\kappa \|_{L^2 (\RR,H)}  +  \|\kappa^\prime \|_{L^2 (\RR,H)} \right)^{1/2}
\end{align*}
where $C=\sqrt{\int_\RR\frac{1}{1+\xi^2 } d\xi}=\sqrt{\pi}$.
In particular, we can apply this to $\widehat{K_S}:\RR\to H$, with $H=S_2 (L^2 (N))$,
the Hilbert space of Hilbert-Schmidt operators on $L^2 (N)$, obtaining the existence of $C>0$
\begin{equation}\label{3}
\sup_{\lambda\in\RR} \| \widehat{K_S}(\lambda)\|^2_{{\rm HS}}\leq C \left(\int_\RR
\| \widehat{K_S}(\lambda)\|^2_{{\rm HS}} d\lambda + \int_\RR
\left \| \frac{d}{d\lambda} \widehat{K_S}(\lambda) \right \|^2_{{\rm HS}} d\lambda \right)
\end{equation}
Thus, there exists $C>0$ such that
\begin{equation}\label{4}
\| S \|^2_{L^2 (\RR_t \times N)}\leq C \left(\int_\RR
\| \widehat{K_S}(\lambda)\|^2_{{\rm HS}} d\lambda + \int_\RR
\left \| \frac{d}{d\lambda} \widehat{K_S}(\lambda)\right \|^2_{{\rm HS}} d\lambda \right).
\end{equation}
Notice that the right hand side is nothing but
$$ C\left( \int_{N\times N \times \RR} |\widehat{K_S} (y,y^\prime,\lambda)|^2 \,dy \,dy^\prime  \,d\lambda
+ \int_{N\times N \times \RR} \left | \frac{d}{d\lambda} \widehat{K_S}(y,y^\prime,\lambda) \right |^2
\, dy\,  dy^\prime\, d\lambda \right)\,.$$
Using elementary properties of the Fourier transform we conclude that the following
Lemma holds true:

\begin{lemma}\label{lem1}
For a translation invariant smoothing operator on $\RR\times N$ with weighted
exponential bounds at infinity we have
\begin{equation}\label{5}
\| S \|^2_{L^2 (\RR \times N)}\leq C\left( \int_{N\times N \times \RR} |K_S (y,y^\prime,t)|^2 (1+t^2)
\, dy\,  dy^\prime\,  dt\right)
\end{equation}
for some universal constant $C$.

\end{lemma}

We can now end the proof of Lemma \ref{lemma:lemma2}.
Our goal is to show that if $\P\in \A$ then
$ \sum_{g\in\Gamma} ||| \P (g) ||| (1+|g|)^{2k}$ is finite. Here $|||\;\;|||$ is the
norm introduced in Definition \ref{def:triple}. With respect to the notation introduced
in that definition we observe that:
 \begin{equation}\label{remark-later}
  P\in A \Rightarrow \chi P\in J\,,\;\;
  [\mathcal{V},P]\in J\,,\;\; \phi [\mathcal{V},P] \in J
  \end{equation}
 We also remark that
  \begin{equation}\label{remark-later-bis}
  P\in A \Rightarrow
   [\phi,P]\in A\,.
    \end{equation}
On the basis of \eqref{remark-later} \eqref{remark-later-bis}
we conclude that
it suffices to show that
\begin{equation}
\P\in\A \Rightarrow  \sum_{g\in\Gamma} \| \P (g) \|(1+|g|)^{2k}
< \infty
\end{equation}
and
\begin{equation}\label{6}\R\in\J \Rightarrow  \sum_{g\in\Gamma} \| \R (g) \|_1 (1+|g|)^{2k}
< \infty
\end{equation}
the latter being in fact the content of Lemma \ref{lemma:lemma2J}.

We now prove \eqref{5}.
To this end we fix a cut-off function near the boundary,
equal to 1 on the boundary and equal to 0 outside a collar neighborhood of the
boundary. Using this cut-off function we can define a section $s: \G\to \A$ to
the indicial homomorphism $I: \A \to \G$. If $\P\in\A$ then we know that we can write
$\P= \P_0 + \P_1$ with $\P_0=s(I(\P))$ and $\P_1\in\J$. Put it differently,
we write $\P$ in terms of its Taylor series at the front face.
We then have
\begin{align*}
& \sum_{g\in\Gamma} \| \P (g) \|(1+|g|)^{2k} \leq C
\sum_{g\in\Gamma} (\| I(\P) (g) \| (1+|g|)^{2k}  + \| \P_ 1 (g)\| (1+|g|)^{2k} )
\end{align*}
We consider the two distinct series
$$
\sum_{g\in\Gamma} \| I(\P) (g) \| (1+|g|)^{2k}  \;\;\text{and}\;\;
\sum_{g\in\Gamma}  \| \P_ 1 (g)\| (1+|g|)^{2k} $$
and we show that they are both convergent; this will suffice.

\noindent
Using  Lemma \ref{lem1}
the term on the left can be bounded by
\begin{align*}
& \sum_{g\in \Gamma}  \int_{N \times N \times \RR_t}
|K_{I(\P)(g)} (y,y^\prime,t)|^2 (1+t^2) (1+|g|)^{2k} \,dy\,dy^\prime \,dt \\
&\leq C \sum_{g\in \Gamma}  \int_{N \times N \times \RR_t}
|K_{I(\P)(g)} (y,y^\prime,t)|^2 \exp \left(\frac{\epsilon}{2} |t| \right) (1+|g|)^{2k} \,dy\,dy^\prime \,dt
\end{align*}
with $N=\pa M_0$.
Now, by assumption, $I(\P)$ is a translation invariant smoothing operator in the $\B^\infty$-Mishchenko-Fomenko
calculus with $\epsilon$-bounds, thus, in particular
$$\sum_{g\in\Gamma}   |K_{I(\P)(g)} (y,y^\prime,t)|^2\exp(\epsilon |t|) (1+|g|)^{2k}
$$
is convergent and uniformly bounded in $N\times N\times \RR$. We can integrate
this series with respect to the finite measure $dy \,dy^\prime \, \exp(-\frac{\epsilon}{2} |t|)dt$
and obtain a finite number; since we can interchange the summation and the integration we conclude that
$$\sum_{g\in \Gamma}  \int_{N \times N \times \RR_t}
|K_{I(\P)(g)} (y,y^\prime,t)|^2 \exp \left(\frac{\epsilon}{2} |t|\right) (1+|g|)^{2k} \,dy\,dy^\prime \,dt < \infty$$
and this implies that
\begin{equation}\label{6.75}
 \sum_{g\in\Gamma} \| I(\P) (g) \| (1+|g|)^{2k} < \infty
\end{equation}
as required.

\noindent
Next we tackle the sum $\sum_{g\in\Gamma}  \| \P_ 1 (g)\| (1+|g|)^{2k} $
or, more generally, the sum $\sum_{g\in\Gamma}  \| \R (g)\| (1+|g|)^{2k} $
for any element $\R$ in $\mathcal{J}$. This is very similar to the closed case
analyzed in \cite{CM}, given that the elements $\R(g)$ are  residual. Indeed,
if $R$ is residual, $R\in J:=\Psi^{-\infty,\epsilon} (M)$, then, in particular, $R$ is a Hilbert-Schmidt operator, in fact, even trace class.
This means that
\begin{equation}\label{7}
\|R\|^2 \leq \| R\|^2_{{\rm HS}} =  \int_{M\times M} |K_R|^2 d{\rm vol}_{M\times M}
\end{equation}
Let now   $e(\epsilon)\in C^\infty (M)$ be  a non vanishing
function equal
to 1 on $M_0$ and equal to $\exp (\epsilon |t|)$ along the cylindrical
end $(-\infty,1]\times \pa M_0$.
Consider  $\R\in \Psi^{-\infty,\epsilon}(M,\B^\infty)$; then, in particular,
$$\sum_{g\in \Gamma}(|K_{\R(g)}|^2  (e(\epsilon)\boxtimes e(\epsilon))) (p,p^\prime))(1+|g|)^{2k},\quad p,p^\prime\in M$$
is convergent and uniformly bounded in $M\times M$. We
can integrate this series with respect to the finite measure
$(e(\epsilon)\boxtimes e(\epsilon))^{-1}d{\rm vol}_{M\times M}$; interchanging summation
and integration we
conclude that
$$\sum_{g\in \Gamma} \int_{M\times M} |K_{\R(g)} |^2 d{\rm vol}_{M\times M} (1+|g|)^{2k}  < \infty$$
Thus
\begin{align*}
\sum_{g\in\Gamma} \|\R(g)\|(1+|g|)^{2k} &
\leq \sum_{g\in\Gamma} \|\R(g)\|_{{\rm HS}}(1+|g|)^{2k}
\\
&=\sum_{g\in \Gamma} \int_{M\times M} |K_{\R(g)} |^2 d{\rm vol}_{M\times M} (1+|g|)^{2k}  < \infty \\
\end{align*}
which is what we wanted to show.
\noindent
Summarizing, we have established \eqref{5}.

\noindent
Regarding \eqref{6}: we know  that
if $R$ is residual then $R$ is trace class. We want to estimate
$\| R \|_1$.
Write $R= ((1+\Delta)^{-\ell} \rho ) (\rho^{-1} (1+\Delta)^{\ell} R)$ with $\ell>\dim M$ and
 $\rho:=e(\epsilon/2)$ (thus $\rho\in C^\infty (M)$ is a non vanishing
function equal
to 1 on $M_0$ and equal to $\exp ((\epsilon/2) |t|)$ along the cylindrical
end $(-\infty,1]\times \pa M_0$).
Then $(1+\Delta)^{-\ell} \rho $ is trace class and we have
$$\| R \|_1 \leq \| (1+\Delta)^{-\ell} \rho )\|_1 \| \rho^{-1} (1+\Delta)^{\ell} R\|\leq C \| \rho^{-1} (1+\Delta)^{\ell} R\|$$
The term $\| \rho^{-1} (1+\Delta)^{\ell} R\|$ can be treated exactly as
above, given that $(1+\Delta)^{\ell} R$ is still residual and the term $\rho^{-1}$
can be absorbed easily in the estimates. Proceeding as above, using the hypothesis
that $\R\in \Psi^{-\infty,\epsilon}(M,\B^\infty)$, we conclude that \eqref{6} holds true.

\bigskip
We are left with the task of proving \eqref{extend-sigma}, i.e. that $\sigma_c$ extends continuously
from $\G_f$ to $\G$. Recall the definition of $\sigma_c$ on $\G_f$. If $B_j\in\G_f$
and
  $\widehat{B}_j (\lambda)=\sum_{g\in \Gamma} \widehat{B}_j (\lambda) (g) g$, $j=0,\dots k+1$,
then
\begin{align*}
&\sigma_c (B_0 +\omega\cdot 1,B_1, \dots, B_{k+1})\\
&:=(-1)^{k+1} \sum_{g_0\cdots g_{k+1}=1}
 \frac{i}{2\pi}\int d\lambda  \Tr \left( (\widehat{B}_0 (\lambda)(g_0)
\cdots \widehat{B}_k (\lambda)(g_k) \frac{d \widehat{B}_{k+1}(\lambda)(g_{k+1})}{d\lambda} \right)
c(g_1,g_2, \dots,g_{k})
\end{align*}
Let
$$
f_j (\lambda,g):= \|  \widehat{B}_k (\lambda)(g) \|\,,j\in\{0,1,\dots,k\} \;\;\text{and}\;\;
f_{k+1} (\lambda,g):=   \left \|  \frac{d \widehat{B}_{k+1}(\lambda)(g)}{d\lambda}\right \|_1 \,.$$
We obtain corresponding elements $f_\ell (\lambda)\in\CC\Gamma$, $\ell\in\{0,1,\dots,k+1\}$. Well-known estimates for the trace-class norm, together with Lemma \ref{lemma:lemma3},
give the existence of $m\in\NN$ such that
\begin{align*}&\sum_{g_0\cdots g_{k+1}} \Tr \left( (\widehat{B}_0 (\lambda)(g_0)
\cdots \widehat{B}_k (\lambda)(g_k) \frac{d \widehat{B}_{k+1}(\lambda)(g_{k+1})}{d\lambda} \right)
c(g_1,g_2, \dots,g_{k})\\
&\leq \nu_m (f_0 (\lambda)) \cdots \nu_m (f_{k+1} (\lambda))
\end{align*}
Easy arguments show that in order to complete the proof of \eqref{extend-sigma} it suffices   to show the following:

\bigskip
\noindent
{\bf Claim:} {\it if $B_j\in\G$ then $\nu_m^2 (f_j (\lambda))$, $j\in\{0,\dots,k+1\}$, is finite and bounded
by $1/(1+\lambda^2)$}.

\bigskip
\noindent
Recall that if $S\in G$ then we have proved the following estimate (see \eqref{3} through \eqref{5}):
\begin{align*}
\sup_{\lambda\in\RR} \| \widehat{K_S}(\lambda)\|_{{\rm HS}}&\leq
C \left(\int_\RR
\| \widehat{K_S}(\lambda)\|^2_{{\rm HS}} d\lambda + \int_\RR
\left \| \frac{d}{d\lambda} \widehat{K_S}(\lambda)\right \|^2_{{\rm HS}} d\lambda \right)\\
&= C\left( \int_{N\times N \times \RR} |\widehat{K_S} (y,y^\prime,\lambda)|^2 \,dy \,dy^\prime  \,d\lambda
+ \int_{N\times N \times \RR} \left | \frac{d}{d\lambda} \widehat{K_S}(y,y^\prime,\lambda) \right |^2
\, dy\,  dy^\prime\, d\lambda \right)\\
&= C\left( \int_{N\times N \times \RR} |K_S (y,y^\prime,t)|^2 (1+t^2)
\, dy\,  dy^\prime\,  dt\right)
\end{align*}
where, as before, we make a small abuse of notation and keep the same symbol for the smoothing kernel
$\widehat{K_S}(\lambda)$ and the smoothing operator it defines on N.
A similar argument shows that, more generally,
\begin{equation}\label{8}
\sup_{\lambda\in\RR}\lambda^{2\ell} \| \widehat{K_S}(\lambda)\|_{{\rm HS}}\leq C \int_{N\times N\times \RR}
|\pa_t^\ell K_S (y,y^\prime,t)|^2   (1+t^2)  dy\,dy^\prime \,dt\,.
\end{equation}
In particular, taking $\ell=0$ and $\ell=1$ and adding we obtain the estimate
\begin{equation}\label{9}
 \| \widehat{K_S}(\lambda)\|_{{\rm HS}}\leq \frac{C}{1+\lambda^2}
\left( \int_{N\times N\times \RR}
|K_S (y,y^\prime,t)|^2   (1+t^2)  dy\,dy^\prime \,dt\,+ \int_{N\times N\times \RR}
|\pa_t K_S (y,y^\prime,t)|^2   (1+t^2)  dy\,dy^\prime \,dt\,.
\right)
\end{equation}
and, hence,
\begin{equation}\label{9.5}
 \| \widehat{K_S}(\lambda)\|_{{\rm HS}}\leq \frac{C}{1+\lambda^2}
\left( \int_{N\times N\times \RR}
|K_S (y,y^\prime,t)|^2   \exp\left(\frac{\epsilon}{2}|t|\right)  dy\,dy^\prime \,dt\,+ \int_{N\times N\times \RR}
|\pa_t K_S (y,y^\prime,t)|^2    \exp\left(\frac{\epsilon}{2}|t|\right)   dy\,dy^\prime \,dt\,.
\right)
\end{equation}
Let now $\mathcal{S}\in\G$, $\mathcal{S}=\sum \mathcal{S}(g) g$, and let $f(\lambda)$ be the function on $\Gamma$
defined by  $f(\lambda)(g):= \| \widehat{\mathcal{S}}(\lambda)(g)\|$. Since $\mathcal{S}$ is a translation-invariant
$\B^\infty$-smoothing
operator with $\epsilon$-bounds we do know that for any $m\in\NN$
$$
\sum_{g\in\Gamma} |K_{\mathcal{S}(g)} (y,y^\prime,t)|^2   \exp(\epsilon |t|)  (1+|g|)^{2m}
+
|\pa_t K_{\mathcal{S}(g)} (y,y^\prime,t)|^2    \exp(\epsilon |t|)  (1+|g|)^{2m}
$$
is convergent and uniformly bounded on $N\times N\times \RR$.
Proceeding precisely as in the steps leading to the proof of \eqref{6.75}
and using \eqref{9.5} we conclude that
the following fundamental estimate holds true:
\begin{equation}\label{10}
\nu_m^2  ( f (\lambda))\equiv \sum_{g\in\Gamma} \| \widehat{\mathcal{S}}(\lambda)(g)\| (1+|g|)^{2m}\leq \frac{C}{1+\lambda^2}
\end{equation}
If now $\mathcal{S}\in\G$ and $h(\lambda)$ is the function on $\Gamma$ defined by
$h(\lambda)(g):= \| \| \widehat{\mathcal{S}}(\lambda)(g)\|_1$ then, similarly,
\begin{equation}\label{11}
\nu_m^2  ( h (\lambda))\equiv \sum_{g\in\Gamma} \| \widehat{\mathcal{S}}(\lambda)(g)\|_1 (1+|g|)^{2m}\leq \frac{C}{1+\lambda^2}
\end{equation}
Indeed, it suffices to observe as before that if $S\in G$ then for $k>\dim N$
$$\|\widehat{S}(\lambda)\|_1 \leq \| (1+\Delta_N)^{-k} \|_1 \|(1+\Delta_N)^k \widehat{S}(\lambda)\|\leq C \|(1+\Delta_N)^k \widehat{S}(\lambda)\|$$
and the  term on the right hand side can be analyzed as before. The proof of the claim, and thus of
Proposition \ref{prop:extension} is now complete.
\end{proof}

\section{The higher Atiyah-Patodi-Singer index formula}\label{sect:theorem}

We are now ready to state and prove the main result of this paper.
Let $c\in Z^k(\Gamma;\CC)$, $k=2p$, be a  normalized group cocycle.
We assume that $\Gamma$ satisfies the (RD)-condition and that $c$
has polynomial growth. We know that under these assumptions
the cyclic cocycle $\tau_c$ extends from $\mathcal{J}_f$ to $\mathcal{J}$ and
our goal is to give a formula for the higher APS
index $$\Ind_{(c,\Gamma)} (\D):=\langle \Ind_\infty(\D),[\tau_c] \rangle$$
with $\Ind_\infty(\D)\in K_0 (\mathcal{J})$ the index class associated to $\D$.

Recall, see Subsection \ref{subsect:excision}, that
if $\A$ and $\G$ are Fr\'echet algebras  and
$I : \A \to \G$ denotes  a  bounded homomorphism, then
the relative group $K_0(\A,\G)$  is by definition $K_0(\A^+,\G^+)$; the latter
 is    the abelian group obtained from equivalence classes
of  triplets
$(e_1, e_0, p_t)$ with
$e_0$ and $e_1$  projections in  $M_{n\times n}(\A^+)$, and
$p_t$  a continuous family of projections in $M_{n\times n}(\G^+)$, $t\in [0,1]$, satisfying
$I (e_i)=p_i$ for $i=0,1$. As already explained, there is a paring
$K_0(\A, \G) \times HC^{2p}(\A, \G) \to \CC$,
which in this case takes the form
$$
\langle [(e_1, e_0, p_t)], [(\tau, \sigma)] \rangle
= {\rm const}_{2p} \left[ \left( \tau(e_1, \dots ,e_1) - \tau(e_0, \dots ,e_0)
-  \sum_{i=0}^{2p} \int_0^1\sigma(p_t,\dots,[\dot{p}_t, p_t],\dots,p_t)dt
\right)\right]
$$
Here ${\rm const}_{2p}:= (-1)^p \frac{(2p)!}{p!}$ and  the commutator appears
at the $i$-th position in the $i$-th summand.
We denote $\tau(e_i, \dots ,e_i)$ simply as $\tau(e_i)$.

We know that associated to a normalized group cocycle $c\in Z^k (\Gamma;\CC)$, $k=2p$, there is a relative cyclic cocycle
$[(\tau^r_c,\sigma_c)]\in HC^{2p} (\A,\G)$ and a relative index class
$\Ind_\infty (\D,\D_{\partial})\in K_0 (\A,\G)$.
We can thus consider, in particular, the pairing
$\langle \Ind_\infty (\D,\D_{\partial}),[(\tau^r_c,\sigma_c)]\rangle$.

Our immediate goal is to show the following crucial identity:
\begin{equation}\label{prot-aps}
\langle \Ind_\infty(\D),[\tau_c] \rangle= \langle \Ind_\infty (\D,\D_{\partial}),[(\tau^r_c,\sigma_c)]\rangle
\end{equation}

The left hand side of formula \eqref{prot-aps} can be written in terms of
the $b$-Connes-Moscovici projector $V^b_{\D^\otimes}$ as
 $$
 \langle [V^b_{\D^\otimes}] -[e_1], \tau_{c} \rangle\,.$$
Recall that if $\beta_{{\rm ex}}: K_0 (\J)\to K_0 (\A,\G)$
is the excision isomorphism then
$$\beta_{{\rm ex}} (  [V^b_{\D^\otimes}] -[e_1])=[V^b_{\D^\otimes},e_1,{\bf c}]\,,$$ with ${\bf c}$ the constant path with value $e_1$.
Since the derivative of the constant path is equal to zero
and since, by its very definition, $\tau^r_{c}|_{\mathcal{J}}=\tau_{c}$,
 we obtain at once the
crucial relation
\begin{equation}\label{restriction-of-reg}
\langle \beta_{{\rm ex}}  (  [V^b_{\D^\otimes}] -[e_1] ), [(\tau^r_{c},\sigma_{c})] \rangle
= \langle     [V^b_{\D^\otimes}] -[e_1], [\tau_{c}] \rangle \,.
\end{equation}
Now we use the excision formula, asserting that
$\beta_{{\rm ex}} ( [V^b_{\D^\otimes}] -[e_1])$
is equal, {\it as a relative class},
to $[V_{\D^\otimes}, e_1, V_{t D^\otimes_{{\rm cyl}}}]$, $t\in[1,+\infty]$. Thus
$$\langle [V_{\D^\otimes}, e_1, V_{t D^\otimes_{{\rm cyl}}}], [(\tau^r_{c},\sigma_{c})] \rangle= \langle   [V^b_{\D^\otimes}] -[e_1], [\tau_{c}] \rangle$$
which is precisely $ \langle \Ind_\infty (\D,\D_{\partial}),[(\tau^r_c,\sigma_c)]\rangle=
\langle \Ind_\infty(\D),[\tau_c] \rangle$. 

We therefore obtain
\begin{align*}
\Ind_{(c,\Gamma)} (\D)&:= \langle \Ind_\infty (\D), [\tau_{c}] \rangle\\
&\equiv  
\langle  [V^b_{\D^\otimes}] -[e_1], [\tau_{c}] \rangle\\
&= 
\langle \beta_{{\rm ex}} ([V^b_{\D^\otimes}] -[e_1]), [(\tau^r_{c},\sigma_{c})] \rangle \\
&= 
\langle [V_{\D^\otimes}, e_1, V_{t D^\otimes_{{\rm cyl}}}], [(\tau^r_{c},\sigma_{c})] \rangle
\end{align*}

In order to compute this expression we recall
the   equality of relative classes
\[ [V_{\D^\otimes}, e_1, V_{t D^\otimes_{{\rm cyl}}}]=
[V_{\D^\otimes}^*, e_1, V_{t D^\otimes_{{\rm cyl}}}^*],\ t\in[1,+\infty].
\]
Therefore if we set $\dvtd:= V_{\D^\otimes}\oplus V_{\D^\otimes}^*$, $\del=e_1\oplus e_1$, $p_t= V_{t D^\otimes_{{\rm cyl}}}\oplus V_{t D^\otimes_{{\rm cyl}}}^*$,
we obtain a relative class $[\dvtd, \del, p_t]$ satisfying
\[
\langle [V_{\D^\otimes}, e_1, V_{t D^\otimes_{{\rm cyl}}}], [(\tau^r_{c},\sigma_{c})] \rangle = \frac{1}{2}\langle [\dvtd, \del, p_t], [(\tau^r_{c},\sigma_{c})] \rangle
\]

Using  the definition of the relative pairing we can thus write 

\begin{align*}
\Ind_{(c,\Gamma)} (\D)&= 
\frac{1}{2}\langle [\dvtd, \del, p_t], [(\tau^r_{c},\sigma_{c})] \rangle\\
&: = \frac{{\rm const}_{2p}}{2} \,\tau^r_{c} (\dvtd -\del)-
\frac{{\rm const}_{2p}}{2} \left[ \sum_{i=0}^{2p}   \int_1^{\infty}\sigma_c (p_t,\dots,[\dot{p}_t, p_t],\dots,p_t)dt
\right]
\,,\\
& =  \frac{{\rm const}_{2p}}{2} \,\tau^r_{c} (\dvtd)-
\frac{{\rm const}_{2p}}{2} \left[ \sum_{i=0}^{2p}   \int_1^{\infty}\sigma_c (p_t,\dots,[\dot{p}_t, p_t],\dots,p_t)dt
\right]
\end{align*}
The convergence at infinity of
 $\left[ \sum_{i=0}^{2p}   \int_1^{\infty}\sigma_c (p_t,\dots,[\dot{p}_t, p_t],\dots,p_t)dt
\right]$
 follows from the fact that the pairing
is well defined but can also be proved directly, using the properties of the heat
kernel and the invertibility of $\D_{{\rm cyl}}$.

Replace now $\D$ by $u\D$, $u>0$.
We obtain, after a simple change of variable in the integral,
$${\rm const}_{2p}\,
\left[ \sum_{i=0}^{2p}   \int_u^{\infty}\sigma_c (p_t,\dots,[\dot{p}_t, p_t],\dots,p_t)dt
\right]
=-2\langle \Ind_\infty (u\D), [\tau_{c}] \rangle
+{\rm const}_{2p} \,\tau^r_{c} (\dvtd)$$
But the absolute pairing  $\langle \Ind_\infty (u\D), [\tau_{c}] \rangle$ in independent of $u$ and equal to $\Ind_{(c,\Gamma)} (\D)$; thus
$${\rm const}_{2p}\,
\left[ \sum_{i=0}^{2p}   \int_u^{\infty}\sigma_c (p_t,\dots,[\dot{p}_t, p_t],\dots,p_t)dt
\right]
=-2\langle \Ind_\infty (\D), [\tau_{c}] \rangle
+{\rm const}_{2p} \,\tau^r_{c} (\dvtd)$$
Now,
by a well-known principle, see \cite[Chapter 8]{Melrose} we know that the short-time behaviour
of the $b$-trace of the heat-kernel
 is computable as in the closed case,
using Getzler-rescaling. Thus, keeping in mind Proposition \ref{prop:short-time},
we can prove that
the second summand
of the right hand side converges
as $u\downarrow 0$ to $2 \int_{M_0} {\rm AS}\wedge \omega_{c}$.
Thus the limit
$$\frac{{\rm const}_{2p}}{2}\,\lim_{u\downarrow 0}
\left[ \sum_{i=0}^{2p}   \int_u^{\infty}\sigma_c (p_t,\dots,[\dot{p}_t, p_t],\dots,p_t)dt
\right]$$
exists
and  is equal to $\int_{M_0} {\rm AS}\wedge \omega_{c}-\Ind_{(c,\Gamma)} (\D)$.

\begin{definition}\label{def:higher-eta}
If  $\Gamma$ satisfies  the (RD) condition and
 $c\in Z^k(\Gamma;\CC)$, $k=2p$ is  a normalized group cocycle (as in \eqref{groupcoc}) of polynomial growth, then we
 define the higher eta invariant associated to $c$ and the boundary operator
 $\D_{\partial}$ as
 \begin{equation}
 \eta_{(c,\Gamma)} (\D_{\partial}):=
 {\rm const}_{2p}\,
\left[ \sum_{i=0}^{2p}   \int_0^{\infty}\sigma_c (p_t,\dots,[\dot{p}_t, p_t],\dots,p_t)dt
\right]
 \end{equation}
 with $p_t= V_{t D^\otimes_{{\rm cyl}}}\oplus V_{t D^\otimes_{{\rm cyl}}}^*$.
\end{definition}


If $N$ is any closed compact manifold, not necessarily a boundary,
and $\Gamma\to \widetilde{N}\to N$ is a Galois $\Gamma$-covering, then
it should be possible to prove, using Getzler rescaling, that the limit
$$\lim_{u\downarrow 0}
\left[ \sum_{i=0}^{2p}   \int_u^{\infty}\sigma_c (p_t,\dots,[\dot{p}_t, p_t],\dots,p_t)dt
\right]
$$
exists. This would allow to define the higher eta invariant  $\eta_{(c,\Gamma)} (\D_{N})$
in general, even for non-bounding
coverings.

\smallskip
\noindent
The arguments given before Definition \ref{def:higher-eta} prove the main result of this paper:

\begin{theorem}\label{theo:main}
Let $\Gamma$ be a finitely generated discrete group satisfying the (RD) condition and
let $c\in Z^k(\Gamma;\CC)$, $k=2p$, be  a normalized group cocycle of polynomial growth.
Let $\Gamma\to \widetilde{M}_0\to M_0$ be a Galois $\Gamma$-covering of
a compact even dimensional manifold with boundary $M_0$, endowed with a Riemannian metric
$g_0$ and a bundle of
unitary Clifford modules $E_0$ with  Clifford connection $\nabla_0$. We assume that all these structures are of product-type near the boundary. Let
$\Gamma\to \widetilde{M}\to M$ be the associated Galois covering with cylindrical ends and let $g$, $E$ and $\nabla$ be the extended structures.
Let $D$ and $\widetilde{D}$ be the associated Dirac operators
and let $\D$ be the operator $D$ twisted by the $\B^\infty$-Mishchenko bundle.
Let us make
the assumption that $\widetilde{D}_{\partial}$ is $L^2$-invertible. Then there is
a well defined higher index $\Ind_{(c,\Gamma)} (\D)$ and the following higher
Atiyah-Patodi-Singer formula holds:

\begin{equation}\label{main}
\Ind_{(c,\Gamma)} (\D)=\int_{M_0} {\rm AS}\wedge \omega_{c} -\frac{1}{2} \eta_{(c,\Gamma)} (\D_{\partial})\,.
\end{equation}
\end{theorem}

Recall now that $\Ind_{(c,\Gamma)} (\D)=\langle \Ch (\Ind_{{\rm MF},\infty}(\D)),t_c \rangle_K $, see \eqref{compatibility};
the Atiyah-Patodi-Singer
formula for the right hand side, proved in \cite[Theorem 12.7]{LPMEMOIRS} and \cite[Appendix]{LPBSMF}, reads
\begin{equation}\label{APS-LP}
\langle \Ch (\Ind_{{\rm MF},\infty}(\D)),t_c \rangle_K =
\int_{M_0} {\rm AS}\wedge \omega_c -\frac{1}{2} \langle \eta_{{\rm Lott}} (\D_{\partial}),t_c
\rangle
\end{equation}
with
$$\eta_{{\rm Lott}} (\D_{\partial})\in
\widehat{\Omega}_* (\B^\infty)/\overline{[\widehat{\Omega}_* (\B^\infty),\widehat{\Omega}_* (\B^\infty)]}$$
 the higher eta invariant of Lott \cite{LottII}.
By using the identity $\Ind_{(c,\Gamma)} (\D)=\langle \Ch (\Ind_{{\rm MF},\infty}(\D)),t_c \rangle_K$ and by comparing the two APS index formulae, we obtain, as a
corollary, the following interesting equality:
\begin{equation}\label{eta-lott=eta-gmp}
\langle \eta_{{\rm Lott}}(\D_{\partial}),t_c
\rangle
= \eta_{(c,\Gamma)} (\D_{\partial})
\,.
\end{equation}
%

{\small \bibliographystyle{plain}
\bibliography{rel-coverings}
}

\end{document}